\documentclass[leqno,11pt]{amsart}

\usepackage{geometry}

\usepackage[all]{xy} 
\usepackage{amsmath, amssymb, amsfonts, latexsym, mdwlist, amsthm, amscd}
\usepackage{subfig}
\usepackage{graphicx}
\usepackage{wrapfig}

\usepackage[bookmarks, colorlinks, breaklinks, pdftitle={Bridgeland Moduli spaces on an Enriques Surface},
pdfauthor={Howard Nuer}]{hyperref}
\hypersetup{linkcolor=blue,citecolor=blue,filecolor=black,urlcolor=blue}

\usepackage{tikz}
\usetikzlibrary{calc,trees,positioning,arrows,chains,shapes.geometric,%
    decorations.pathreplacing,decorations.pathmorphing,shapes,%
    matrix,shapes.symbols}

\tikzset{
>=stealth',
  punktchain/.style={
    rectangle,
    rounded corners,
    draw=black, thick,
    minimum height=3em,
    text centered,
    on chain},
  line/.style={draw, thick, <-},
  element/.style={
    tape,
    top color=white,
    bottom color=blue!50!black!60!,
    minimum width=8em,
    draw=blue!40!black!90, very thick,
    text width=10em,
    minimum height=3.5em,
    text centered,
    on chain},
  every join/.style={->, thick,shorten >=1pt},
  decoration={brace},
  tuborg/.style={decorate},
  tubnode/.style={midway, right=2pt},
}

\usepackage{paralist}
\setdefaultenum{(a)}{(i)}{}{}
\usepackage{enumitem} 

\def\C{\ensuremath{\mathbb{C}}}

\def\H{\ensuremath{\mathbb{H}}}

\def\P{\ensuremath{\mathbb{P}}}
\def\Q{\ensuremath{\mathbb{Q}}}
\def\R{\ensuremath{\mathbb{R}}}
\def\Z{\ensuremath{\mathbb{Z}}}

\def\alg{\mathrm{alg}}
\def\Amp{\mathrm{Amp}}
\def\Aut{\mathop{\mathrm{Aut}}\nolimits}

\def\ch{\mathop{\mathrm{ch}}\nolimits}

\def\Coh{\mathop{\mathrm{Coh}}\nolimits}
\def\codim{\mathop{\mathrm{codim}}\nolimits}

\def\deg{\mathop{\mathrm{deg}}}

\def\dim{\mathop{\mathrm{dim}}\nolimits}

\def\inf{\mathop{\mathrm{inf}}\nolimits}

\def\Ext{\mathop{\mathrm{Ext}}\nolimits}
\def\ext{\mathop{\mathrm{ext}}\nolimits}
\def\lExt{\mathop{\mathcal Ext}\nolimits} 
\def\Fix{\mathop{\mathrm{Fix}(\iota)}}
\def\For{\mathop{\mathrm{Forg}_G}}

\def\GL{\mathop{\mathrm{GL}}}
\def\Hal{H^*_{\alg}}

\def\HN{\mathop{\mathrm{HN}}\nolimits}
\def\Hom{\mathop{\mathrm{Hom}}\nolimits}

\def\RlHom{\mathop{\mathbf{R}\mathcal Hom}\nolimits}

\def\Inf{\mathop{\mathrm{Inf}_G}}

\def\Jac{\mathop{\mathrm{Jac}}\nolimits}

\def\mod{\mathop{\mathrm{mod}}\nolimits}

\def\Nef{\mathrm{Nef}}
\def\num{\mathop{\mathrm{num}}\nolimits}
\def\Num{\mathop{\mathrm{Num}}\nolimits}
\def\NS{\mathop{\mathrm{NS}}\nolimits}

\def\Pic{\mathop{\mathrm{Pic}}}

\def\rk{\mathop{\mathrm{rk}}}
\def\Sch{\mathop{\mathrm{Sch}}\nolimits}
\def\Sing{\mathop{\mathrm{Sing}}}

\def\SSL{\mathop{\mathrm{SSL}}}
\def\ST{\mathop{\mathrm{ST}}\nolimits}

\def\Sym{\mathop{\mathrm{Sym}}\nolimits}

\def\td{\mathop{\mathrm{td}}\nolimits}

\def\v{\mathop{\pi^*v}\nolimits}

\def\MG13{\ensuremath{{\mathcal M}_{\Gamma_1(3)}}}
\def\tildeMG13{\ensuremath{\widetilde{\mathcal M}_{\Gamma_1(3)}}}
\def\Stab{\mathop{\mathrm{Stab}}}
\def\into{\ensuremath{\hookrightarrow}}
\def\onto{\ensuremath{\twoheadrightarrow}}

\def\blank{\underline{\hphantom{A}}}

\def\Db{\mathrm{D}^{\mathrm{b}}}

\def\pt{[\mathrm{pt}]}

\newcommand\stv[2]{\left\{#1\,\colon\,#2\right\}}

\makeatletter
\newtheorem*{rep@theorem}{\rep@title}
\newcommand{\newreptheorem}[2]{%
\newenvironment{rep#1}[1]{%
 \def\rep@title{#2 \ref{##1}}%
 \begin{rep@theorem}}%
 {\end{rep@theorem}}}
\makeatother

\newtheorem{Thm}{Theorem}[section]
\newreptheorem{Thm}{Theorem}
\newtheorem{Prop}[Thm]{Proposition}
\newtheorem{PropDef}[Thm]{Proposition and Definition}
\newtheorem{Lem}[Thm]{Lemma}

\newtheorem{Cor}[Thm]{Corollary}
\newreptheorem{Cor}{Corollary}

\newreptheorem{Con}{Conjecture}

\newtheorem{Obs}[Thm]{Observation}

\newtheorem{thm-int}{Theorem}

\theoremstyle{definition}
\newtheorem{Def-s}[Thm]{Definition}
\newtheorem{Def}[Thm]{Definition}
\newtheorem{Rem}[Thm]{Remark}

\def\C{\ensuremath{\mathbb{C}}}

\def\G{\ensuremath{\mathbb{G}}}
\def\H{\ensuremath{\mathbb{H}}}

\def\P{\ensuremath{\mathbb{P}}}
\def\Q{\ensuremath{\mathbb{Q}}}
\def\R{\ensuremath{\mathbb{R}}}
\def\Z{\ensuremath{\mathbb{Z}}}

\def\AA{\ensuremath{\mathcal A}}

\def\CC{\ensuremath{\mathcal C}}

\def\EE{\ensuremath{\mathcal E}}
\def\FF{\ensuremath{\mathcal F}}
\def\GG{\ensuremath{\mathcal G}}
\def\HH{\ensuremath{\mathcal H}}

\def\NN{\ensuremath{\mathcal N}}
\def\OO{\ensuremath{\mathcal O}}
\def\PP{\ensuremath{\mathcal P}}
\def\RR{\ensuremath{\mathcal R}}
\def\SS{\ensuremath{\mathcal S}}

\def\TT{\ensuremath{\mathcal T}}

\def\ZZ{\ensuremath{\mathcal Z}}

\def\Y{\ensuremath{\tilde{Y}}}

\def\MMM{\mathfrak M}

\def\MY{\ensuremath{\mathfrak M_{\sigma,Y}(v)}}
\def\MYs{\ensuremath{\mathfrak M^s_{\sigma,Y}(v)}}
\def\MX{\ensuremath{\mathfrak M_{\sigma',\Y}(\v)}}
\def\MXs{\ensuremath{\mathfrak M^s_{\sigma',\Y}(\v)}}
\def\mY{\ensuremath{M_{\sigma,Y}(v)}}
\def\mYs{\ensuremath{M^s_{\sigma,Y}(v)}}
\def\mYss{\ensuremath{M^{ss}_{\sigma,Y}(v)}}
\def\mX{\ensuremath{M_{\sigma',\Y}(\v)}}
\def\mXs{\ensuremath{M^s_{\sigma',\Y}(\v)}}

\newcommand{\ignore}[1]{}

\begin{document}

\title[Bridgeland Moduli spaces on an Enriques Surface]{Projectivity and Birational Geometry of Bridgeland Moduli spaces on an Enriques Surface}

\author{Howard Nuer}
\address{Department of Mathematics, Rutgers University, 110 Frelinghuysen Rd., Piscataway, NJ 08854, USA}
\email{hjn11@math.rutgers.edu}
\urladdr{http://math.rutgers.edu/~hjn11/}



\begin{abstract}
We construct moduli spaces of semistable objects on an Enriques surface for generic Bridgeland stability condition and prove their projectivity.  We further generalize classical results about moduli spaces of semistable sheaves on an Enriques surface to their Bridgeland counterparts.  Using Bayer and Macr\`{i}'s construction of a natural nef divisor varying with the stability condition, we begin a systematic exploration of the relation between wall-crossing on the
Bridgeland stability manifold and the minimal model program for these moduli spaces.  We give three applications of our machinery to obtain new information about the classical moduli spaces of Gieseker-stable sheaves:

1) We obtain a region in the ample cone of the moduli space of Gieseker-stable sheaves which works for all unnodal Enriques surfaces.

2) We determine the nef cone of the Hilbert scheme of $n$ points on an unnodal Enriques surface in terms of the classical geometry of its half-pencils and the Cossec-Dolgachev $\phi$-function.

3) We recover some classical results on linear systems on Enriques surfaces and obtain some new ones about $n$-very ample line bundles.
\end{abstract}

\vspace{-1em}

\maketitle

\setcounter{tocdepth}{1}
\tableofcontents

\section{Introduction}\label{sec:intro}

\subsection{Background}

For many years now there has been interest in various kinds of moduli spaces of coherent sheaves on smooth projective varieties and their geometry.  The most classical of these moduli spaces are the moduli spaces of slope stable vector bundles and of Gieseker-stable sheaves, respectively, and each is constructed as a GIT quotient.  As each of these notions of stability depends on the choice of an ample divisor on the underlying smooth projective variety, a natural approach to studying the (birational) geometry of these moduli spaces is via VGIT (Variation of GIT), but for dimension reasons this cannot give all of the birational geometry of these moduli spaces.  A revolutionary approach to this problem came from Bridgeland's definition of the notion of a stability condition on a triangulated category \cite{Bri07}, an attempt at a mathematical definition of Douglas's $\Pi$-stability \cite{Dou02} for $D$-branes in string theory.  Bridgeland proved there that, when nonempty, the set of (full numerical) Bridgeland stability conditions on a triangulated category $\TT$ forms a complex manifold of dimension $\rk K_{\num}(\TT)$.  

Both moduli spaces of stable sheaves and Bridgeland stability have been most studied in the case when the underlying variety is a smooth projective surface $X$, in which case the triangulated category for Bridgeland stability is $\Db(X)$, the bounded derived category of coherent sheaves.  To that end, the most studied examples have been on $\P^2$ and projective K3 and Abelian surfaces.  In each case, the space of Bridgeland stability conditions, $\Stab(X)$, has been shown to be nonempty and to admit a wall and chamber decomposition in the following sense: the set of $\sigma$-semistable objects with some fixed numerical invariants is constant in each chamber of the decomposition, while crossing a wall necessarily changes the stability of some object in this set.  

For the numerical invariants of a Gieseker semistable sheaf, Bridgeland identified in \cite{Bri08} a certain chamber of the corresponding wall and chamber decomposition, the Gieseker chamber, where the set of $\sigma$-semistable objects can be identified with the Gieseker semistable sheaves with respect to a generic polarization.  He conjectured there that each chamber of $\Stab(X)$ should admit a coarse moduli space of $\sigma$-semistable objects and that crossing a wall should induce a birational transformation between the moduli spaces corresponding to adjacent chambers.  In this way, Bridgeland forged a new tool for the investigation of the birational geometry of moduli spaces of Gieseker-stable sheaves in a way that expands and parallels the Hassett-Keel program for $\overline{M}_{g,n}$, in which other minimal models of a given moduli space are expected to be modular themselves and parametrize objects stable under a slightly different sense of stability.

The picture envisioned by Bridgeland has been confirmed in examples, first to a certain extent by Arcara and Bertram in \cite{AB13}, but most notably by Arcara, Bertram, Coskun, and Huizenga in \cite{ABCH13},  for the Hilbert scheme of points on $\P^2$, and then by Bayer and Macr\`{i} in \cite{BaMa,BaMa13} for all numerical invariants on K3 surfaces.  Following these ground-breaking developments, there has been an explosion of activity surrounding the use of Bridgeland stability conditions to study the MMP (minimal model program) of moduli spaces of Gieseker-stable sheaves on $\P^2$, Del Pezzo surfaces, K3 surfaces, and Abelian surfaces.

In this paper, the first in a two-part series, we extend the investigation to the classical cousins of K3 surfaces, the Enriques surfaces, defined as quotients of K3 surfaces by fixed-point free involutions.  

\subsection{Summary of main results}

This paper grew out of an attempt to understand the techniques of \cite{BaMa, BaMa13} and to generalize them to Enriques surfaces in the hope of obtaining similar results, while at the same time investigating the subtle differences between the geometry of Enriques and K3 surfaces.  Our first result is, as one would hope, a direct generalization of Theorem 1.3(a) in \cite{BaMa} about the existence and projectivity of Bridgeland stable moduli spaces:

\begin{Thm} Let $Y$ be an unnodal Enriques surface, $v\in \Hal(Y,\Z)$ a Mukai vector, and $\sigma\in\Stab^{\dagger}(Y)$ a stability condition generic with respect to $v$ (i.e. it does not lie on any wall).  Then the coarse moduli space $\mY$ of $\sigma$-semistable objects with Mukai vector $v$ exists as a $K$-trivial normal projective variety.
\end{Thm}

Let $X$ be a smooth projective variety satisfying the conclusion of the above theorem, namely the existence of coarse moduli spaces of $\sigma$-stable objects.  In order to investigate the relationship between wall-crossing on $\Stab(X)$ and the birational geometry of Bridgeland moduli spaces, we make use of the ``Bayer-Macr\`{i}" map, $\ell:\overline{\CC}\to\Nef(M_{\sigma,X}(v))$, where $\sigma$ is in the interior of the chamber $\CC$, defined as follows:  fix a (quasi-)universal family $\EE$ on $M_{\sigma,X}(v)$; then for any integral curve $C\subset M_{\sigma,X}(v)$ and $\sigma'\in\overline{\CC}$, Bayer and Macr\`{i} set $$\ell_{\sigma',\EE}([C])=\ell_{\sigma',\EE}.C:=\Im(-\frac{Z(\Phi_{\EE}(\OO_C))}{Z(v)}),$$ where $\Phi_{\EE}:\Db(M_{\sigma,X}(v))\to \Db(X)$ is the integral transform with kernel $\EE$.  They show the following fundamental result \cite[Theorem 1.1]{BaMa}:

\begin{Thm} The association $C\mapsto \ell_{\sigma',\EE}([C])$ gives a well-defined nef numerical divisor class $\ell_{\sigma',\EE}\in N^1(M_{\sigma,X}(v))$.  Moreover, $\ell_{\sigma',\EE}.C=0$ if and only if the objects corresponding to two general points of $C$ are $S$-equivalent with respect to $\sigma'$.  
\end{Thm}

Two natural questions are whether or not the image of the restriction to the interior, $\ell_{\EE}:\CC\to \Nef(M_{\sigma,X}(v))$, is contained in the ample cone $\Amp(M_{\sigma,X}(v))$, and what happens at the walls of $\CC$.  We answer these questions with our next result in the case $X$ is an unnodal Enriques surface $Y$, generalizing Theorems 1.3(b) and 1.4 in \cite{BaMa}.  To place the result in context, let $W$ be a wall of the chamber decomposition for $v$, $\sigma_0$ a generic point of $W$ (so that it does not lie on any other walls), and $\sigma_{\pm}$ two nearby stability conditions, one on each side of the wall, with corresponding moduli spaces $M_{\pm}:=M_{\sigma_{\pm},Y}(v)$ and (quasi-)universal families $\EE_{\pm}$.  Then by the closedness of semistability in $\Stab(Y)$ (from the definition of the topology on $\Stab(Y)$), $\EE_{\pm}$ provide families of $\sigma_0$-semistable objects and $\ell_{\sigma_0,\EE_{\pm}}$ a nef divisor class on their respective moduli spaces.  Our next main result is this:

\begin{Thm}Let $Y$ be an unnodal Enriques surface and $v\in\Hal(Y,\Z)$ a Mukai vector.
\begin{enumerate}
\item Suppose $\sigma\in\Stab^{\dagger}(Y)$ is generic with respect to $v$.  Then $\ell_{\sigma}$ is ample.

Now suppose that $v$ is primitive, then:
\item The divisor classes $\ell_{0,\pm}$ are semiample (and remain so when restricted to each component of $M_{\pm}$), and they induce contraction morphisms $$\pi_{\pm}:M_{\pm}\rightarrow Z_{\pm},$$ where $Z_{\pm}$ are normal projective varieties.

\item Suppose that $M_{\sigma_0,Y}^s(v)\neq\varnothing$ (so that in particular $\ell_{0,\pm}$ is big as well).  
\begin{itemize}
\item If either $\ell_{0,\pm}$ is ample, then the other is ample, and the birational map $$f_{\sigma_0}:M_+\dashrightarrow M_-$$ obtained by crossing the wall in $\sigma_0$ extends to an isomorphism.
\item If $\ell_{0,\pm}$ are not ample and the complement of $M_{\sigma_0,Y}^s(v)$ has codimension at least 2, then $f_{\sigma_0}:M_+\dashrightarrow M_-$ is the flop induced by $\ell_{0,+}$.  More precisely, we have a commuative diagram of birational maps
\begin{equation*}
\xymatrix{ M_{\sigma_{+},Y}(v)\ar@{-->}[rr]^{f_{\sigma_0}}\ar[dr]_{\pi_+} && M_{\sigma_{-},Y}(v)\ar[dl]^{\pi_{-}}\\
& Z_+=Z_{-} &
},
\end{equation*}
and $f_{\sigma_0}^*\ell_{0,-}=\ell_{0,+}$.
\end{itemize}
\end{enumerate}
\end{Thm}

Unlike in the K3 case, where for primitive $v$ and $\sigma$ generic $M_{\sigma,X}(v)$ is a hyperk\"{a}hler (irreducible holomorphic symplectic) manifold whose Beauville-Donagi form on $H^2(M_{\sigma,X}(v),\Z)$ is isometric via the Donaldson morphism with the restriction of the Mukai pairing to $v^{\perp}\subset \Hal(X)$, we do not have such strong tools in the case of Enriques surfaces.  Instead, we make use of the natural pull-back morphism $$\Phi:\mY\to\mX,$$ where $\pi:\tilde{Y}\to Y$ is the K3 universal cover of $Y$, and stability on the two surfaces is related via the notion of ``induced stability condition" introduced by Macr\`{i}, Mehrotra, and Stellari in \cite{MMS}.  We show that $\Phi$ is a finite morphism onto a Lagrangian subvariety of $\mX$, which allows us to transfer many of the results of \cite{BaMa} from $\tilde{Y}$ to $Y$.  We believe that at the wall, $\ell_{\sigma_0,\pm}$ is always big and $\pi_{\pm}$ always birational.  We have arguments in many cases, but this issue will be pursued in \cite{Nue14b} where it fits in most naturally with the results we have obtained on the classification of walls in terms of their induced birational transformation.

\subsection{Applications}

We end the paper with three sections on applications.  In the first section on the classical moduli spaces of Gieseker stable sheaves, we obtain explicit, effective bounds on the Gieseker chamber and thus, via the Bayer-Macr\`{i} map, on the ample cone of these moduli spaces.  

In the second section, we use the Bridgeland stability techniques developed in this paper to describe explicitly $\Nef(Y^{[n]})$, the nef cone of the Hilbert scheme of $n$ points on $Y$, in terms of the beautiful geometry of Enriques surfaces and their elliptic pencils.  More specifically, denote by $2B$ the divisor parametrizing the locus of non-reduced 0-dimensional subschemes of length $n$ on $Y$, and for every $H\in\Amp(Y)$, denote by $\tilde{H}$ the locus of 0-dimensional subschemes of length $n$ on $Y$ meeting a member of the linear system $|H|$.  Then $B$ and $\langle \tilde{H} |H\in\Amp(Y)\rangle$ generate $\Pic(Y^{[n]})$.  We recall the $\phi$-function defined in \cite[Section 2.7]{CD} by $$\phi(D)=\inf\{|D.F|:F\in\Pic(Y),F^2=0\},$$ for $D^2>0$.  The significance of primitive $F\in\Pic(Y)$ with $F^2=0$ is that either $F$ or $-F$ is effective, say $F$, and $2F$ defines an elliptic pencil on $Y$ with exactly two multiple fibres $F$ and $F+K_Y$, respectively.  As suggested by the definition of $\phi$, these ``half-pencils" govern much of the geometry of $Y$.  We obtain the following result which confirms this overarching theme in the study of Enriques surfaces:

\begin{Thm}\label{hilbert nef cone} Let $Y$ be an unnodal Enriques surface and $n\geq 2$.  Then $\tilde{D}-aB\in\Nef(Y^{[n]})$ if and only if $D\in\Nef(Y)$ and $0\leq na\leq D.F$ for every $0<F\in\Pic(Y)$ with $F^2=0$, or in other words $0\leq a\leq \frac{\phi(D)}{n}$. Moreover, the face given by $a=0$ induces the Hilbert-Chow morphism, and for every ample $H\in\Pic(Y)$, $\tilde{H}-\frac{\phi(H)}{n}B$ induces a flop.
\end{Thm}

Recall that the Hilbert-Chow morphism $h:Y^{[n]}\to Y^{(n)}$ sends a 0-dimensional subscheme of length $n$ to its underlying 0-cycle and is a divisorial contraction with exceptional locus $2B$.  Moreover, we can describe explicitly the flop induced by $\tilde{H}-\frac{\phi(H)}{n}B$ as follows:  for every half-pencil $F$ such that $H.F=\phi(H)$, we get a pair of disjoint codimension $n$ components of the exceptional locus isomorphic to $F^{[n]}$ and $(F+K_Y)^{[n]}$, respectively, where these precisely parametrize the sublocus of $n$-points on $Y$ contained in $F$ and $F+K_Y$, respectively.  On the component corresonding to $F$, say, the contracted fibers (i.e. the curves of $S$-equivalent objects) are exactly the fibers of the natural Abel-Jacobi morphism $F^{[n]}\to \Jac^n(F)\cong F$ ($g(F)=1$) associating to $Z$ the line bundle $\OO_F(Z)$, where the objects of $F^{[n]}$ fit into the destabilizing exact sequence 

$$0\to\OO(-F)\to I_Z\to\OO_F(-Z)\to 0.$$  Upon crossing the wall, we perform a Mukai-type flop, replacing the $\P^{n-1}$-bundle $F^{[n]}$ over the base $F$ with another one parametrizing objects sitting in an exact sequence of the form $$0\to \OO_F(-Z)\to E\to \OO(-F)\to 0,$$ where $E\in\P(\Ext^1(\OO(-F),\OO_F(-Z))).$

Theorem \ref{hilbert nef cone} can be seen as giving an alternative definition of the $\phi$-function, and we believe it likely that many of its properties can be recovered from the convexity of $\Nef(Y^{[n]})$ and pairing divisors with test curves.  

In the final section, we apply Theorem \ref{hilbert nef cone} to recover a weak form of a classical result about linear systems on unnodal Enriques surfaces (see \cite[Theorems 4.4.1 and 4.6.1]{CD}):
\begin{Cor} Let $Y$ be an unnodal Enriques surface and $H\in\Pic(Y)$ ample with $H^2=2d$.  Then
\begin{enumerate}
\item The linear system $|H|$ is base-point free if and only if $\phi(H)\geq 2$,
\item If $|H|$ is very ample, then $\phi(H)\geq 3$.  Conversely, if $\phi(H)\geq 4$ or $\phi(H)=3$ and $d=5$, then $|H|$ is very ample.
\item The linear system $|2H|$ is base-point free and $|4H|$ is very ample.
\end{enumerate}
\end{Cor}

\begin{Rem} The stronger form of the above result says that for $d\geq 5$, $|H|$ is very ample if and only if $\phi(H)\geq 3$.  It follows that for any ample $H$, even $|3H|$ is very ample. We believe the Bridgeland stability methods we use to prove our weakened version above can be pushed further to prove the full result (see Remark \ref{full result}).
\end{Rem}

We also obtain some new results about $n$-very ample line bundles on unnodal Enriques surfaces.  Recall that a line bundle $\OO_X(H)$ on a smooth projective surface $X$ is called $n$-very ample if the restriction map $$\OO_X(H)\to\OO_Z(H)$$ is surjective for every 0-dimensional subscheme  $Z$ of length $n+1$.  Then we prove the following result:
\begin{Cor} Let $Y$ be an unnodal Enriques surface and $H\in\Pic(Y)$ ample with $H^2=2d$.  Then $\OO_Y(H)$ is $n$-very ample provided that $$0\leq n\leq \frac{d\cdot\phi(H)}{2d-\phi(H)}-1.$$
\end{Cor}

Both of these results follow from the following vanishing theorem which is a direct consequence of the Bridgeland stability techniques of Theorem \ref{hilbert nef cone}:
\begin{Prop}\label{vanishing} Let $Y$ be an unnodal surface and $H\in\Pic(Y)$ ample with $H^2=2d$.  Then for any $Z\in Y^{[n]}$, $$H^i(Y,I_Z(H+K_Y))=0,\text{ for }i>0,$$ provided that $$1\leq n< \frac{d\cdot \phi(d)}{2d-\phi(d)}.$$
\end{Prop}

\subsection{Open questions} A fundamental and basic question is whether or not the moduli spaces $\mY$ are irreducible.  In the case of odd rank it is known that $\mY$ has two isomorphic irreducible components destinguished by whether $\det(E)=c_1(v)$ or $c_1(v)+K_Y$.  It is unknown for even rank, though H. Kim has conjectured it to be true for moduli spaces of Gieseker stable sheaves, which would imply the result for Bridgeland moduli spaces using the techniques developed here.  

Additionally, the bigness of $\ell_{\sigma_0,\pm}$ in the case of a totally semistable wall, that is $M^s_{\sigma_0,Y}(v)=\varnothing$, is unknown at the moment.  The failure of bigness would give the existence of interesting fibration structures on $M_{\pm}$ provided by the morphisms $\pi_{\pm}$.   Another open question is whether or not one can classify entirely and in total generality the walls in $\Stab^{\dagger}(Y)$ in terms of the geometry of the morphism $\pi_{\pm}$ for any given primitive Mukai vector $v$.  Even if this is achieved, we wonder if a full Hassett-Keel-type result holds true for Enriques surfaces as shown to be the case for K3 surfaces in \cite{BaMa13}.  That is, does every minimal model of $\mY$ appear after deformation of the stability condition, i.e. as another moduli space of Bridgeland stable objects?  We take up both of these questions in \cite{Nue14b}.

It is natural to wonder what kind of varieties the moduli spaces $\mY$ are.  For primitive $v$, these are smooth $K$-trivial projective varieties in the sense that their canonical divisors are torsion in the Picard group.  We suspect, based on examples, that these are always of Calabi-Yau-type, being genuine Calabi-Yau manifolds for $v$ of even rank while only admitting genuine Calabi-Yau \'{e}tale covers in the case of odd rank.

Finally, we have made the assumption that $Y$ is unnodal in most of the results of this paper, and one would like to extend our program to the case of nodal Enriques surfaces, both for the sake of completeness and since these surfaces can be studied much more explicitly.  We suspect that they behave quite differently due to the presence of spherical objects.

\subsection{Outline of the paper}  We review Bridgeland stability conditions in general in Section \ref{sec:Bridgeland} and in the specific case of K3 and Enriques surfaces in Section \ref{sec:reviewK3}.  In Section \ref{sec:ReviewGieseker} we review important background on moduli spaces of Gieseker-stable sheaves on K3 and Enriques surfaces.

In Sections \ref{sec:ModuliStacks}-\ref{sec:Projectivity}, we prove the existence and projectivity of the moduli spaces $\mY$: we first use results of \cite{Lie}, \cite{Tod08} to prove that the algebraic stacks $\MY$ parametrizing $\sigma$-semistable objects of Mukai vector $v$ are non-empty finite-type Artin stacks, $\C^*$-gerbes over an algebraic space in the case of primitive $v$.  Then we use the pull-back morphism $\Phi$ to deduce projectivity of the underlying coarse moduli spaces.  

In Section \ref{sec:SingKod} we study singularities and show that the moduli spaces $\mY$ are Gorenstein, normal projective varieties with only l.c.i. canonical singularities and torsion canonical divisor, smooth in the case of primitive $v$, generalizing results of Kim \cite{Kim}, Sacca \cite{Sac}, and Yamada \cite{Yamada} for the classical Gieseker moduli spaces.  

We extend the fundamental results of \cite{BaMa} about the image of the Bayer-Macr\`{i} map $\ell:\CC\to N^1(\mY)$ and prove our second main theorem about wall-crossing and briational geometry in Sections \ref{sec:BM divisor} and \ref{sec: wall-crossing}.

We use the machinery we develop in this paper to prove our explicit bound for the Gieseker chamber in Section \ref{sec: stable sheaves}, generalizing \cite[Corollary 9.14]{BaMa}, and prove Theorem \ref{hilbert nef cone} about the Hilbert scheme of points in Section \ref{sec: hilbert scheme}.

\subsection{Acknowledgements}The auther would like to thank his advisor, Lev Borisov, for his constant support and guidance.  He would also like to thank Arend Bayer and Emanuele Macr\`{i} in particular for their support and very helpful discussions.  Our presentation largely follows their foundational paper \cite{BaMa}.  The author also benefitted from and would like to thank I. Coskun, I. Dolgachev, A. Maciocia, E. Markman, and G. Sacca for very helpful discussions.

This project was first inspired by the talks of Arend Bayer and Emanuele Macr\`{i} at the Graduate Student Workshop on Moduli Spaces and Bridgeland Stability at UIC and the latter's encouragement and enthusiasm for the author to undertake this project.  He thanks the organizers of the workshop for fostering such a productive learning environment.  Finally, the author was partially supported by NSF grant DMS 1201466.

\subsection*{Notation and Convention} \label{subsec:notation}

For an abelian group $G$ and a field $k(=\Q,\R,\C)$, we denote by $G_k$ the $k$-vector space $G\otimes k$.

Throughout the paper, when unspecified $X$ will denote a smooth projective variety over $\C$.  For a (locally-noetherian) scheme (or algebraic space) $S$, we will use the notation $\Db(S)$ for its bounded derived category of coherent sheaves, $\mathrm{D}_{qc}(S)$ for the unbounded derived category of quasi-coherent sheaves, and $\mathrm{D}_{S\text{-}\mathrm{perf}}(S\times X)$ for the category of $S$-\emph{perfect complexes}.
(An $S$-perfect complex is a complex of $\OO_{S\times X}$-modules which locally, over $S$, is quasi-isomorphic to a bounded complex of coherent shaves which are flat over $S$.)

We will abuse notation and denote all derived functors as if they were underived.
We denote by $p_S$ and $p_X$ the two projections from $S\times X$ to $S$ and $X$, respectively.
Given $\EE\in\mathrm{D}_{qc}(S\times X)$, we denote the Fourier-Mukai functor associated to $\EE$ by
\[
\Phi_\EE (\blank):= (p_X)_*\left( \EE\otimes p_S^*(\blank)\right).
\]

We let $K_{\num}(X)$ be the numerical Grothendieck group of $X$ and denote by $\chi(-)$ (resp., $\chi(-,-)$) the Euler characteristic on $K_{\num}(X)$: for $E,F\in\Db(X)$,
\[ \chi(E)=\sum_p (-1)^p\, h^p(X,E) \quad \text{and} \quad
\chi(E,F)=\sum_p (-1)^p\, \mathrm{ext}^p(E,F).
\]

We denote by $\mathrm{NS}(X)$ the N\'eron-Severi group of $X$, and write $N^1(X):=\mathrm{NS}(X)_\R$.
The space of full numerical stability conditions on $\Db(X)$ will be denoted by $\Stab(X)$.

Given a complex $E\in\Db(X)$, we denote its cohomology sheaves by $\HH^*(E)$.
The skyscraper sheaf at a point $x\in X$ is denoted by $k(x)$.
For a complex number $z\in\mathbb{C}$, we denote its real and imaginary part by $\Re z$ and $\Im z$, respectively.

We call a variety $X$ \textit{K-trivial} if its canonical divisor $K_X$ is numerically trivial.

\section{Review: Bridgeland stability conditions}\label{sec:Bridgeland}

In this section, we give a brief review of stability conditions on derived categories,
as introduced in \cite{Bri07}.

Let $X$ be a smooth projective variety, and denote by $\Db(X)$ its bounded derived category of coherent sheaves.
A \emph{full numerical stability condition} $\sigma$ on $\Db(X)$ consists of a pair $(Z,\AA)$, where
$Z\colon K_{\num}(X)\to\C$ 
is a group homomorphism (called the \emph{central charge}) and $\AA\subset\Db(X)$ 
is the \emph{heart of a bounded t-structure}, satisfying the following three properties:
\begin{enumerate}
\item For any $0 \neq E\in\AA$ the central charge $Z(E)$ lies in the following semi-closed
upper half-plane:
\begin{equation} \label{eq:Zpositivity}
Z(E) \in \H := \HH \cup \R_{<0} = \R_{>0} \cdot e^{(0,1]\cdot i\pi}
\end{equation}
\suspend{enumerate}
One can think of this condition as two separate positivity conditions: $\Im Z$ defines a rank function
on the abelian category $\AA$, i.e., a non-negative function $\rk \colon \AA \to \R_{\ge 0}$
that is additive on short exact sequences.
Similarly, $-\Re Z$ defines a degree function $\deg \colon \AA \to \R$, which has the property that
$\rk(E) = 0 \Rightarrow \deg(E) > 0$. 
We can use them to define a notion of slope-stability with respect
to $Z$ on the abelian category $\AA$ via the slope $\mu(E) = \frac{\deg(E)}{\rk(E)}$: an object $E$ is called \emph{semistable} (resp. \emph{stable}) if every proper subobject $0\neq F\subset E$ satisfies $\mu(F)\leq\mu(E)$ (resp. $\mu(F)<\mu(E)$).
\resume{enumerate}
\item With this notion of slope-stability, every object in $E \in \AA$ has a Harder-Narasimhan
filtration $0 = E_0 \into E_1 \into \dots \into E_n = E$ such that the $E_i/E_{i-1}$'s are $Z$-semistable,
with $\mu(E_1/E_0) > \mu(E_2/E_1) > \dots > \mu(E_n/E_{n-1})$.
\item There is a constant $C>0$ such that, for all $Z$-semistable object $E\in \AA$, we have
\begin{align*}
\lVert E \rVert \le C \lvert Z(E) \rvert,
\end{align*}
where $\lVert \ast \rVert$ is a fixed norm on $K_{\num}(X)\otimes\R$.
\end{enumerate}
This last condition is often called the \emph{support property} and is equivalent to Bridgeland's notion of a \emph{full} stability condition.

\begin{Def} \label{def:algebraic}
A stability condition is called \emph{algebraic} if its central charge takes values in $\Q\oplus\Q
\sqrt{-1}$.
\end{Def}
As $K_{\num}(X)$ is finitely generated, for an algebraic stability condition the 
image of $Z$ is a discrete lattice in $\C$.

Given $(Z, \AA)$ as above, one can extend the notion of stability to $\Db(X)$ as follows:
for $\phi \in (0, 1]$,
we let $\PP(\phi) \subset \AA$ be the full subcategory of $Z$-semistable objects
with $Z(E) \in \R_{>0} e^{i\phi\pi}$; for general $\phi$, it is defined
by the compatibility $\PP(\phi + n) = \PP(\phi)[n]$.
Each subcategory $\PP(\phi)$ is extension-closed and abelian. Its nonzero objects are called
$\sigma$-\emph{semistable} of phase $\phi$, and its simple objects are called
$\sigma$-\emph{stable}. Then each object $E \in \Db(X)$ has a 
\emph{Harder-Narasimhan filtration}, where the inclusions $E_{i-1} \subset E_i$ are replaced by
exact triangles $E_{i-1} \to E_i \to A_i$, and where the $A_i$'s are $\sigma$-semistable of decreasing phases
$\phi_i$.  The
category $\PP(\phi)$ necessarily has finite length. Hence every object in $\PP(\phi)$ has a finite
Jordan-H\"older filtration, whose filtration quotients are $\sigma$-stable objects of the phase
$\phi$.  Two objects $A,B\in\PP(\phi)$ are called $S$-\emph{equivalent} if their Jordan-H\"older
factors are the same (up to reordering).  We define the \emph{mass} of an object $E$ for a given $\sigma$  by $m_{\sigma}(E)=\sum_i |Z_{\sigma}(A_i)|$, where $A_i$ are the $\sigma$-semistable factors of $E$.  Of course, it follows that $|Z(E)|\leq m_{\sigma}(E)$.  We sometimes abuse notation and write $(Z,\PP)$ in place of $(Z,\AA)$.

The set of stability conditions will be denoted by $\Stab(X)$.
It has a natural metric topology (see \cite[Prop.\ 8.1]{Bri07} for the explicit form of
the metric). Bridgeland's main theorem is the following:

\begin{Thm}[Bridgeland] \label{thm:Bridgeland-deform}
The map
\begin{align*}
\ZZ \colon \Stab(X) \to \Hom(K_{\num}(X), \C), \qquad (Z, \AA)\mapsto Z, 
\end{align*}
is a local homeomorphism.
In particular, $\Stab(X)$ is a complex manifold of finite dimension equal to the rank of $K_{\num}(X)$.
\end{Thm}
In other words, a stability condition $(Z, \AA)$ can be deformed uniquely given a small deformation
of its central charge $Z$.

\begin{Rem}\label{rmk:GroupAction}
There are two group actions on $\Stab(X)$, see \cite[Lemma 8.2]{Bri07}:
the group of autoequivalences $\Aut(\Db(X))$ acts on the left via
$\Pi(Z,\AA)=(Z\circ\Pi_*^{-1},\Pi(\AA))$, where $\Pi \in \Aut(\Db(X))$ and
$\Pi_*$ is the automorphism induced by $\Pi$ at the level of numerical Grothendieck groups.
We will often abuse notation and denote $\Pi_*$ by $\Pi$, when no confusion arises.
The universal cover $\widetilde\GL^+_2(\R)$ of the group 
$\GL^+_2(\R)$ of matrices with positive determinant acts on the right as a lift of the
action of $\GL^+_2(\R)$ on
$\Hom(K_{\num}(X), \C) \cong \Hom(K_{\num}(X), \R^2)$. We typically only use the action of the
subgroup $\C \subset \widetilde\GL^+_2(\R)$ given as the universal cover of $\C^* \subset \GL^+_2(\R)$: given
$z \in \C$, it acts on $(Z, \AA)$ by $Z \mapsto e^{2\pi i z}\cdot Z$, and by modifying
$\AA$ accordingly.
\end{Rem}

\section{Review: Moduli spaces for stable sheaves on K3 and Enriques surfaces}\label{sec:ReviewGieseker}

In this section we give a summary of stability for sheaves on Enriques surfaces and their K3 universal covers.
We start by recalling the basic lattice-theoretical structure, given by the Mukai lattice.
We then review slope and Gieseker stability and the existence and non-emptiness of moduli spaces of semistable sheaves.

\subsection*{The algebraic Mukai lattice}
Let $\tilde{Y}$ be a smooth projective K3 surface.
We denote by $H^*_{\alg}(\tilde{Y},\Q)$ the algebraic part of the whole cohomology of $\tilde{Y}$, namely
\begin{equation}\label{eq:AlgebraicMukaiLattice}
H^*_\alg(\tilde{Y},\Z) = H^0(\tilde{Y},\Z) \oplus \mathrm{NS}(\tilde{Y}) \oplus H^4(\tilde{Y},\Z).
\end{equation}  Suppose the Enriques surface $Y$ is obtained as the quotient of $\tilde{Y}$ by a fixed-point free involution $\iota$, then we may similarly define the algebraic part of the cohomology of $Y$, except that it is the entire cohomology in this case, so we drop the subscript for this and use it slightly differently below.

Let $v \colon K_{\num}(\tilde{Y}) \xrightarrow{\sim} H^*_{\alg}(\tilde{Y}, \Z)$ be the Mukai vector given by $v(E) = \ch(E) \sqrt{\td(\tilde{Y})}$, i.e. \[v(E)=(r(E),c_1(E),r(E)+ch_2(E)),\] where we write the Mukai vector according to the decomposition \eqref{eq:AlgebraicMukaiLattice}.  The Mukai vector on the Enriques surface is defined the same and is similarly given by \[v(E)=(r(E),c_1(E),\frac{r(E)}{2}+ch_2(E)),\] inducing an isomorphism $v\colon K_{\num}(Y)\xrightarrow{\sim} H^*_{\alg}(Y,\Z):=H^0(Y,\Z)\oplus \mathrm{NS}(Y)\oplus \frac{1}{2}\Z \rho_Y\subset H^{*}(Y,\Q)$, where $\rho_Y$ is the fundamental class of $Y$.
We denote the Mukai pairing $H^*_{\alg}(\tilde{Y}, \Z) \times H^*_{\alg}(\tilde{Y}, \Z) \to \Z$ (resp. $H^*_{\alg}(Y,\Z)\times H^*_{\alg}(Y,\Z)\to \Z$) by $(\blank, \blank)$; it can be
defined by $(v(E), v(F)) := - \chi(E, F)$.  
According to the decomposition \eqref{eq:AlgebraicMukaiLattice}, we have
\[
\left( (r,c,s),(r',c',s')\right) = c.c' - rs' - r's,
\]
for $(r,c,s),(r',c',s')\in H^*_\alg(\tilde{Y},\Z)$ (resp. $H^*_{\alg}(Y,\Z)$).  

Given a Mukai vector $v\in H^*_{\alg}(\tilde{Y}, \Z)$(resp. $H^*_{\alg}(Y,\Z)$), we denote
its orthogonal complement by
\[
v^\perp:=\left\{w\in H^*_{\alg}(\tilde{Y}, \Z)(\text{resp. }w\in H^*_{\alg}(Y,\Z))\colon (v,w)=0 \right\}.
\]
We call a Mukai vector $v$ \emph{primitive} if it is not divisible in $H^*_\alg(\tilde{Y},\Z)$(resp. $H^*_\alg(Y,\Z)$).  Note that the covering space map $\pi:\Y\rightarrow Y$ induces an embedding $$\pi^*:\Hal(Y,\Z)\into\Hal(\Y,\Z)$$ such that $(\v,\pi^*w)=2(v,w)$, and it identifies $\Hal(Y,\Z)$ with an index 2 sublattice of the $\iota^*$-invariant component of $\Hal(\tilde{Y},\Z)$.  The embedding $\pi^*:\NS(Y)\into\NS(\tilde{Y})$ is primitive, however, and identifies $\NS(Y)$ with the $\iota^*$-invariant part of $\NS(\tilde{Y})$. It follows that for a primitive Mukai vector $v\in\Hal(Y,\Z)$, $\v$ is divisible by at most 2.  All of this is encapsulated nicely in the following lemma \cite[Lemma 2.5]{Hau10}:

\begin{Lem} Let $v=(r,c_1,s)\in\Hal(Y,\Z)$ be a primitive Mukai vector.  Then $\gcd(r,c_1,2s)=1$ or $2$.  Moreover:
\begin{itemize}
\item if $\gcd(r,c_1,2s)=1$, then either $r$ or $c_1$ is not divisible by 2 (i.e. $\v$ is primitive);
\item if $\gcd(r,c_1,2s)=2$, then $c_2$ must be odd and $r+2s\equiv 2 \mod 4$ (i.e. $\v$ is divisible by 2).
\end{itemize}
\end{Lem}

In particular, for odd rank Mukai vectors or Mukai vectors with $c_1$ primitive, $\v$ is still primitive, while primitive Mukai vectors with $\gcd(r,c_1)=2$ (and thus necessarily $\gcd(r,c_1,2s)=2$) must satisfy $v^2\equiv 0 \mod 8$, as can be easily seen.

\subsection*{Slope stability}
Let $\omega,\beta\in\NS(X)_\Q$ with $\omega$ ample on a smooth projective surface $X$.
We define a slope function
$\mu_{\omega, \beta}$ on $\Coh X$ by
\begin{equation} \label{eq:muomegabeta}
\mu_{\omega, \beta}(E) = 
\begin{cases}
\frac{\omega.(c_1(E) - r(E)\beta)}{r(E)} & \text{if $r(E) > 0$,} \\
+\infty & \text{if $r(E) = 0$.}
\end{cases}
\end{equation}
This gives a notion of slope stability for sheaves, for which Harder-Narasimhan filtrations exist (see \cite[Section 1.6]{HL}).  
We will sometimes use the notation $\mu_{\omega,\beta}$-stability.

\subsection*{Gieseker stability}
Let $\omega,\beta\in\NS(X)_\Q$ with $\omega$ ample.
We define the \emph{twisted Hilbert polynomial} by
\[
P(E,m) := \int_X e^{-\beta} . (1,m\omega,\frac{m^2\omega^2}{2}) . v(E),
\]
for $E\in\Coh(X)$.
This gives rise to the notion of $\beta$-twisted $\omega$-Gieseker stability for sheaves, introduced first in \cite{MW97}.
When $\beta=0$, this is nothing but Gieseker stability.
We refer to \cite[Section 1]{HL} for basic properties of Gieseker stability.

\subsection*{Moduli spaces of stable sheaves}
Let $\omega,\beta\in\NS(X)_\Q$ with $\omega$ ample.
We fix a Mukai vector $v\in H^*_{\alg}(X,\Z)$ (or in other words, we fix the topological invariants $r,c_1,c_2$).
We denote by $\MMM_{\omega}^{\beta}(v)$ the moduli stack of flat families of $\beta$-twisted $\omega$-Gieseker semistable sheaves with Mukai vector $v$.
By \cite[Section 4]{HL} and \cite{MW97}, there exists a projective variety $M_{\omega}^{\beta}(v)$ which is a coarse moduli space parameterizing $S$-equivalence classes of semistable sheaves.
The open substack $\MMM_{\omega}^{\beta,s}(v)\subseteq \MMM_{\omega}^{\beta}(v)$ parameterizing stable sheaves is a $\mathbb{G}_m$-gerbe over the open subset $M_{\omega}^{\beta, s}(v)\subseteq M_{\omega}^{\beta}(v)$.
When $\beta=0$, we will denote the corresponding objects by $\MMM_\omega(v)$, etc.

The following is the main result on moduli spaces of stable sheaves on K3 surfaces $\tilde{Y}$.
In its final form it is proved by Yoshioka in \cite[Theorems 0.1 \& 8.1]{Yos01}.
We start by recalling the notion of positive vector, following \cite[Definition 0.1]{Yos01}.

\begin{Def}\label{def:YoshiokaPositive}
Let $v_0=(r,c,s)\in H^*_{\alg}(\tilde{Y}, \Z)$ be a primitive class.
We say that $v_0$ is \emph{positive} if $v_0^2\geq-2$ and
\begin{itemize}
\item either $r>0$,
\item or $r=0$, $c$ is effective, and $s\neq0$,
\item or $r=c=0$ and $s>0$.
\end{itemize}
\end{Def}

\begin{Thm}[Yoshioka]\label{thm:YoshiokaNonEmptyness}
Let $v\in H^*_{\alg}(\tilde{Y}, \Z)$.
Assume that $v=mv_0$, with $m\in\Z_{>0}$ and $v_0$ a primitive positive vector.
Then $M_{\omega}^{\beta}(v)$ is non-empty for all $\omega,\beta$.
\end{Thm}

\begin{Rem}\label{rmk:SmoothnessFactorialityModuliSpace}
We keep the assumptions of Theorem \ref{thm:YoshiokaNonEmptyness}.
We further assume that $\omega$ is \emph{generic} with respect to $v$ so that stable factors of a semistable sheaf $E$ with $v(E)=v$ must have Mukai vector $m'v_0$ for $m'<m$.
\begin{enumerate}
\item\label{enum:KLS} By \cite{KLS06}, $M_{\omega}^{\beta}(v)$ is then a normal irreducible projective variety with $\Q$-factorial singularities.
\item\label{enum:DefoEquivalent} If $m=1$, then by \cite{Yos01} $M_{\omega}^{\beta,s}(v)=M_{\omega}^{\beta}(v)$ is a smooth projective irreducible symplectic manifold of dimension $v^2+2$, deformation equivalent to the Hilbert scheme of points on a K3 surface.
\end{enumerate}
\end{Rem}

Now let us recall the relevant analogues of the above results for Enriques surfaces.  First recall that for a variety $X$ over $\C$, the cohomology with compact support $H^*_c(X,\Q)$ has a natural mixed Hodge structure.  Let $e^{p,q}(X):=\sum_k (-1)^k h^{p,q}(H^k_c(X))$ and $e(X):=\sum_{p,q}e^{p,q}(X)x^p y^q$ be the virtual Hodge number and Hodge polynomial, respectively.  Moreover, for an Enriques surface $Y$ we recall that the kernel of $\NS(Y)\to \Num(Y)$ is given by $\langle K_Y\rangle$, and thus $$M_{\omega}(v)=M_{\omega}(v,L_1)\coprod M_{\omega}(v,L_2),$$ where $M_{\omega}(v,L_i)$ denotes those $E\in M_{\omega}(v)$  with $\det(E)=L_i$ and $L_2=L_1(K_Y)\in \Pic(Y)$ so $c_1=c_1(L_1)=c_2(L_2)\in \Num(Y)$.  Finally, let us recall the following definition:

\begin{Def} A smooth projective surface $X$ is called \emph{unnodal} if it contains no curves of negative self-intersection and \emph{nodal} otherwise.
\end{Def}

Thus an Enriques surface $Y$ is unnodal if it contains no $(-2)$-curves, and these are generic in their moduli space.  An important consequence of this is that the ample cone is entirely round, i.e. $D\in\Pic(Y)$ is ample if and only if $D^2>0$ and it intersects some effective curve positively.  The following result is proved in \cite{Yos03}:

\begin{Thm}\label{yosh odd} Let $v=(r,c,s)\in H^*_{\alg}(Y,\Z)$ be a primitive Mukai vector with $r$ odd and $Y$ unnodal.  Then \[e(M_{\omega}(v,L))=e(Y^{[\frac{v^2+1}{2}]}),\] for a general $\omega$ and $L\in \Pic(Y)$ satisfies $c_1(L)=c_1$.  In particular, 
\begin{itemize}
\item $M_{\omega}(v)\neq \varnothing$ for a general $\omega$ if and only if $v^2\geq -1$.
\item $M_{\omega}(v,L)$ is irreducible for general $\omega$.
\end{itemize}  
\end{Thm}

For even rank Mukai vectors, Hauzer proved the following in \cite{Hau10}:
\begin{Thm}\label{hauz even} Let $Y$ be an unnodal Enriques surface and $v=(r,c,s)\in H^*_\alg(Y,\Z)$ a primitive Mukai vector with $r$ even.  Then for a general polarization $\omega$ we have \[e(M_{\omega}(v,L))=e(M_{\omega}((r',c_1',-s'/2),L')),\] where $r'$ is 2 or 4.
\end{Thm}

Non-emptiness of $M_{\omega}(v)$ with $v^2\geq -1$ was proved in \cite{Kim06} for the case $r(v)=2$ and in \cite{Nue14a} for the case $r(v)=4$.  We can summarize the above discussion as follows:

\begin{Thm}\label{nonemptiness Enriques} For a general polarization $\omega$ on an unnodal Enriques surface $Y$ and primitive $v\in\Hal(Y,\Z)$ such that $v^2\geq -1$, $M_{\omega}(v)\neq \varnothing$.  If $r(v)$ is odd, then it consists of two isomorphic irreducible components.
\end{Thm}

We expect these moduli spaces to be irreducible in the even rank case, but we do not need this for the sequel.  Let us just recall the following result in this direction from \cite{Nue14a}:

\begin{Thm} Let $v=(r,c,s)$ be a primitive Mukai vector on an unnodal Enriques surface $Y$ with $v^2=0$.  Then $M_{H,Y}(v)$ is a smooth irreducible elliptic curve if $\gcd(r,c,2s)=1$ or isomorphic to $Y$ itself if $\gcd(r,c,2s)=2$.
\end{Thm}

\section{Review: Stability conditions on K3 and Enriques surfaces}
\label{sec:reviewK3}

In this section we give a brief review of Bridgeland's results on stability conditions for K3
surfaces in \cite{Bri08}, and of results by Toda, Yoshioka and others related to moduli spaces of
Bridgeland-stable objects.

\subsection*{Space of stability conditions for a K3 surface}\label{stab on K3}
Let $\tilde{Y}$ be a smooth projective K3 surface.
Fix $\omega,\beta\in\NS(\tilde{Y})_\Q$ with $\omega$ ample.

Let $\TT(\omega,\beta) \subset \Coh \tilde{Y}$ be the subcategory of torsion sheaves and torsion-free sheaves whose HN-filtrations factors (with respect to slope-stability)
have $\mu_{\omega, \beta} > 0$, and $\FF(\omega,\beta)$ the subcategory of torsion-free sheaves with
HN-filtration factors satisfying $\mu_{\omega, \beta} \le 0$.
Next, consider the abelian category
\begin{equation*} 
\AA(\omega,\beta):=\left\{E\in\Db(\tilde{Y}):\begin{array}{l}
\bullet\;\;\HH^p(E)=0\mbox{ for }p\not\in\{-1,0\},\\\bullet\;\;
\HH^{-1}(E)\in\FF(\omega,\beta),\\\bullet\;\;\HH^0(E)\in\TT(\omega,\beta)\end{array}\right\}
\end{equation*}
and the $\C$-linear map
\begin{equation} \label{eq:ZK3}
Z_{\omega,\beta} \colon K_{\num}(\tilde{Y})\to\C,\qquad E\mapsto(\exp{(\beta+\sqrt{-1}\omega)},v(E)).
\end{equation}
If $Z_{\omega,\beta}(F)\notin\R_{\leq0}$ for any spherical sheaf $F\in\Coh(\tilde{Y})$ (e.g., this holds
when $\omega^2>2$), then by \cite[Lemma 6.2, Prop.\ 7.1]{Bri08},
the pair $\sigma_{\omega,\beta}=(Z_{\omega,\beta},\AA(\omega,\beta))$ defines a stability condition.
For objects $E \in \AA(\omega, \beta)$, we will denote their phase with respect to $\sigma_{\omega,
\beta}$ by $\phi_{\omega, \beta}(E) = \phi(Z(E)) \in (0, 1]$.
By using the support property, as proved in \cite[Proposition 10.3]{Bri08}, we can extend the above and define stability conditions $\sigma_{\omega,\beta}$, for $\omega,\beta\in\NS(\tilde{Y})_\R$.

Denote by $U(\tilde{Y})\subset\Stab(\tilde{Y})$ the open subset consisting of the stability conditions
$\sigma_{\omega,\beta}$ just constructed up to the action of $\widetilde{\GL}_2(\R)$. It can also be
characterized as the open subset $U(\tilde{Y}) \subset \Stab(\tilde{Y})$ consisting of stability conditions for
which the skyscraper sheaves $k(x)$ of points are stable of the same phase.
Let $\Stab^{\dagger}(\tilde{Y}) \subset \Stab(\tilde{Y})$ be the connected component containing $U(\tilde{Y})$.
Let $\PP(\tilde{Y})\subset H^*_{\alg}(\tilde{Y})_\C$ be the subset consisting of vectors whose real and imaginary
parts span positive definite two-planes in $H^*_{\alg}(\tilde{Y})_\R$ with respect to the Mukai pairing.
It has two connected components, corresponding to the induced orientation of the two-plane.
Choose $\PP^+(\tilde{Y})\subset\PP(X)$ as the connected component containing the vector
$(1,i\omega,-\omega^2/2)$, for $\omega\in\mathrm{NS}(\tilde{Y})_\R$ the class of an ample divisor.
Furthermore, let $\Delta(\tilde{Y}):=\{ s \in H^*_{\alg}(\tilde{Y}, \Z) \colon s ^2=-2\}$ be the set of spherical 
classes, and, for $s \in\Delta$,
\[
s ^\perp_\C:=\stv{ \Omega\in H^*_{\alg}(\tilde{Y})_\C}{(\Omega,s)=0}.
\]
Finally, set
\[
\PP_0^+(\tilde{Y}):=\PP^+(\tilde{Y}) \setminus \underset{s\in\Delta(\tilde{Y})}{\bigcup}s^\perp_\C\subset H^*_\alg(\tilde{Y})_\C.
\]

Since the Mukai pairing $(\blank,\blank)$ is non-degenerate, we can define
$\eta(\sigma)\in H^*_{\alg}(\tilde{Y})_\C$ 
for a stability condition $\sigma=(Z,\PP)\in\Stab^\dagger(\tilde{Y})$ by 
\[
\ZZ (\sigma) (\blank) = \left(\blank,\eta(\sigma)\right).
\]

\begin{Thm}[Bridgeland]\label{thm:BridgelandK3}
The map $\eta\colon \Stab^\dagger(\tilde{Y})\to H^*_{\alg}(\tilde{Y})_\C$ is a covering map onto its image $\PP^+_0(\tilde{Y})$.
\end{Thm}

\subsection*{Space of stability conditions for an Enriques surface via induction}
Let $\pi:\tilde{Y}\rightarrow Y$ denote the covering map of an Enriques surface $Y$ by its covering K3 $\tilde{Y}$.  Via the fixed-point free covering involution $\iota$, $\Coh(Y)$ is naturally isomorphic to the category of coherent $G$-sheaves on $\tilde{Y}$, $\Coh_G(\tilde{Y})$, where $G=\langle \iota^*\rangle$, thus giving a natural equivalence of $\Db(Y)$ with $\Db_G(\tilde{Y})$.  We make this identification implicitly below.  

In \cite{MMS} the authors construct two faithful adjoint functors $$\For:\Db_G(\tilde{Y})\rightarrow \Db(\tilde{Y}),$$ which forgets the $G$-sheaf structure, and $$\Inf:\Db(\tilde{Y})\rightarrow \Db_G(\tilde{Y})$$ defined by $$\Inf(E):=\oplus_{g\in G} g^*E.$$  Under the above identifications we have $\For=\pi^*$ and $\Inf=\pi_*$.  Since $G$ acts on $\Stab(\tilde{Y})$ via the natural action of $\Aut(\Db(\tilde{Y}))$ on $\Stab(\tilde{Y})$, we can define $$\Gamma_{\tilde{Y}}:=\{\sigma\in \Stab(\tilde{Y}):g^*\sigma=\sigma,\text{ for all }g\in G\}.$$

They define two induced continuous maps.  First, they define $(\pi^*)^{-1}:\Gamma_{\tilde{Y}}\rightarrow \Stab(\Db(Y))$ given by $Z_{(\pi^*)^{-1}(\sigma)}=Z_{\sigma}\circ \pi^*$ and $\mathcal P_{(\pi^*)^{-1}(\sigma)}(\phi)=\{E\in \Db(Y) : \pi^* E\in \mathcal P_{\sigma}(\phi)\}$, where we use $\pi^*$ also for the morphism between $K$-groups. Second, they define $(\pi_*)^{-1}:(\pi^*)^{-1}(\Gamma_{\tilde{Y}})\rightarrow \Stab(\tilde{Y})$ similarly with $\pi^*$ replaced by $\pi_*$.

We consider the connected component $\Stab^{\dagger}(\tilde{Y})\subset \Stab(\tilde{Y})$ described in the section above.  The following result \cite[Proposition 3.1]{MMS} is relevant to us:

\begin{Thm}\label{induced} The non-empty subset $\Sigma(Y):=(\pi^*)^{-1}(\Gamma_{\tilde{Y}}\cap \Stab^{\dagger}(\tilde{Y}))$ is open and closed in $\Stab(Y)$, and it is embedded into $\Stab^{\dagger}(\tilde{Y})$ as a closed submanifold via the functor $(\pi_*)^{-1}$.  Moreover, the diagram 

\[\begin{CD}
\Gamma_{\tilde{Y}}\cap \Stab^{\dagger}(\tilde{Y})@ >(\pi^*)^{-1} >> \Sigma(Y)@ >(\pi_*)^{-1} >>\Gamma_{\tilde{Y}}\cap\Stab^{\dagger}(\tilde{Y})\\
@VV V  @VV\mathcal Z V  @VV V\\
({K_{\num}(\tilde{Y})}_{\C})_G^{\vee} @>(\pi^*)^{\vee} >>{K_{\num}(Y)}_{\C}^{\vee}@>\pi_*^{\vee} >>({K_{\num}(\tilde{Y})}_{\C})_G^{\vee}
\end{CD}\] commutes.

\end{Thm}

Now we denote by $\Stab^{\dagger}(Y)$ the (non-empty) connected component of $\Sigma(Y)$ containing the images via $(\pi^*)^{-1}$ of the stability conditions $(Z_{\omega,\beta},\mathcal A(\omega,\beta))$ defined above with $G$-invariant $\omega,\beta\in \NS(\tilde{Y})_{\Q}$ (thus giving $G$-invariant stability conditions).  It is worth recalling that by \cite[Remark 3.2]{MMS} $\Stab^{\dagger}(Y)$ can alternatively be described by repeating the construction of $\Stab^{\dagger}(\tilde{Y})$ but for the Enriques surface $Y$.  In particular, there is a connected open subset $U(Y)\subset \Stab(Y)$ with $U(Y)\cap \mathcal Z^{-1}(\mathcal P_0^+(Y))\neq \varnothing$ consisting of stability conditions $\sigma$ such that the structure sheaves of points are stable in $\sigma$ of the same phase.  Here we define $\mathcal P_0^+(Y):=(\pi^*)^{\vee}\mathcal P_0^+(\tilde{Y})_G$, where $(-)_G$ denotes taking the $G$-invariant part.  We take the group $\Aut^0(\Db(Y))$ of those autoequivalences preserving $\Sigma(Y)$ and inducing the identity on cohomology via the homomorphism $\Pi$ constructed in \cite{MMS}.  They proved the following result analogous to Theorem \ref{thm:BridgelandK3} above:

\begin{Prop} The map $\eta\colon\Sigma(Y)\rightarrow {K_{\num}(Y)}_{\C}$ defines a covering map onto $\mathcal P_0^+(Y)$ such that $\Aut^0(\Db(Y))/\langle(-)\otimes \omega_Y\rangle$ acts as the group of deck transformations.
\end{Prop}

\subsection*{The Wall-and-Chamber structure}

A key ingredient in the connection between the stability manifold and the birational geometry of Bridgeland moduli spaces is the existence of a wall-and-chamber structure on $\Stab(X)$.  For a fixed $\sigma\in \Stab(X)$, we say a subset $\mathcal S\subset \Db(X)$ has \emph{bounded mass} if there exists $m>0$ such that $m_{\sigma}(E)\leq m$ for any $E\in \mathcal S$.  It follows from the definition of the metric topology on $\Stab(X)$ that being of bounded mass is independent of the specific initial stability condition $\sigma$ and depends only on the connected component it lies on.  We have the following general result (see \cite[Proposition 2.8]{Tod08}):

\begin{Prop}\label{pseudo-walls} Let $X$ be a smooth projective variety.  Assume that for any bounded mass subset $\mathcal S\subset \Db(X)$ the set of numerical classes $$\{[E]\in K_{\num}(X)|E\in \mathcal S\}$$ is a finite.  Then for any compact subset $B\subset \Stab^*(X)$ (an arbitrary connected component of $\Stab(X)$), there exists a finite number of real codimension one submanifolds $\{W_{\gamma}|\gamma\in \Gamma\}$ on $\Stab^*(X)$ such that if $\Gamma'$ is a subset of $\Gamma$ and  $$\mathcal C\subset \bigcap_{\gamma\in \Gamma'}(B\cap W_{\gamma})\backslash \bigcup_{\gamma \notin \Gamma'} W_{\gamma}$$ is one of the connected components, then if $E\in \mathcal S$ is semistable for some $\sigma\in \mathcal C$, then it is semistable for all $\sigma\in \mathcal C$.
\end{Prop}

We now verify the assumption of the proposition when $X=Y$ is an Enriques surface:

\begin{Lem}\label{bounded mass} Suppose the subset $\mathcal S\subset \Db(Y)$ has bounded mass in a connected component $\Stab^*(Y)$ of $\Stab(Y)$ which intersects $\mathcal Z^{-1}(\mathcal P_0^+(Y))$.  Then the set of numerical classes $\{[E]|E\in \mathcal S\}$ is finite.
\end{Lem}
\begin{proof} Since the conclusion is true for bounded mass subsets $\mathcal S'\subset \Db(\tilde{Y})$ for the covering K3 $\tilde{Y}$ above, we first show that $\pi^*(\mathcal S)$ is of bounded mass.  Indeed by assumption there is a $\sigma=(Z,\mathcal P)\in \Stab^*(Y)$ such that $\mathcal Z(\sigma)\in \mathcal P_0^+(Y)$, and by the definition of $\Sigma(Y)$ and the commutativity of the diagram in Theorem \ref{induced} we can lift $\sigma$ to a $\sigma'=(Z',\mathcal P')\in\Gamma_{\tilde{Y}}\cap \Stab^{\dagger}(\tilde{Y})$ such that $\mathcal Z(\sigma')\in\mathcal P_0^+(\tilde{Y})_G$.  Now by our assumption about $\mathcal S$ there exists $m>0$ such that $m_{\sigma}(E)\leq m$ for any $E\in \mathcal S$.  For any $E\in\mathcal S$, the proof of \cite[Lemma 2.8]{MMS} shows that the HN-filtration of $\pi^*E$ in $\sigma'$ is the image via $\pi^*$ of the HN-filtration in $\sigma$.  Then $m_{\sigma}(E)=m_{\sigma'}(\pi^*E)$ from this and the definition of the induction of stability conditions.  This shows that $\mathcal S'=\pi^*(\mathcal S)$ is of bounded mass.

It follows that $$\{[F]\in K_{\num}(\tilde{Y})|F\in \mathcal S'\}$$ is a finite set.  But if $F=\pi^*E$, then $[F]=\pi^*[E]$, and $\pi^*$ is an isomorphism onto ${K_{\num}(\tilde{Y})}_G$, so the set $$\{[E]\in K_{\num}(Y)|E\in \mathcal S\}$$ is finite.
\end{proof}

For the remainder of this section, we let $X$ denote any smooth projective variety satisfying the assumption Proposition \ref{pseudo-walls}, though for our purposes $X=Y$ or $\tilde{Y}$.  It is worthwhile now to point out the following fact which is crucial in considering the stability of objects as $\sigma$ varies:
\begin{Lem}\label{finite m vectors} Given a subset of $\SS\subset \Db(X)$ of bounded mass and a compact subset $B$, then $$\{v(E)|E\in\SS\text{ or is a (semi)stable factor of some }E'\in\SS\text{ for some }\sigma\in B\}$$ is a finite set.
\end{Lem}
\begin{proof} We essentially follow \cite[Proposition 9.3]{Bri08}.  Let $T$ be the set of nonzero objects $A\in\Db(X)$ such that for some $\sigma\in B$ and some $E\in\SS$, $m_{\sigma}(A)\leq m_{\sigma}(E)$.  Then the fact that $B$ is compact implies that the quotient $m_{\tau}(E)/m_{\sigma}(E)$ is uniformly bounded for all nonzero $E\in \Db(X)$, and for all $\sigma,\tau\in B$, so the subset $T$ has bounded mass since $\SS$ does.  By Lemma \ref{bounded mass}, the set of numerical classes (and thus Mukai vectors) of elements of $T$ is finite.

The final observation that is important here is that if $A$ is a semistable factor of an object $E\in\SS$ for some $\sigma\in B$, then $A$ is in $T$, as are its stable factors, since $m_{\sigma}(E)=\sum_i m_{\sigma}(A_i)$ with the sum ranging over all of the stable factors of its semistable factors.
\end{proof}

The proof of Proposition \ref{pseudo-walls}, i.e. the construction of the $W_{\gamma}$, entails considering, for each of the finitely many pairs $\gamma=(v_i,v_j)$ of linearly independent Mukai vectors of stable factors of objects in $\SS$, the real codimension one submanifold $$W_{\gamma}=\{\sigma=(Z,\AA)\in\Stab^*(X)|Z(v_i)/Z(v_j)\in \R_{>0}\}.$$  The most important consequence of this construction is that when $\sigma\in \CC$, a Mukai vector with the same phase as the Mukai vector $v$ of an object in $\SS$ must lie on the ray $\R_{>0}v$.  All of this tells us that we can construct a wall-and-chamber structure on $\Stab^{\dagger}(Y)$ for any bounded mass subset $\SS$.  Following \cite{Mac}, we call the codimension one submanifolds from Proposition \ref{pseudo-walls} \emph{pseudo-walls} for the bounded mass subset $\SS$.

Let us now fix a class $v \in K_{\num}(X)$, and consider the set $\SS$ of
$\sigma$-semistable objects $E \in \Db(X)$ of class $v$ as $\sigma$ varies.  This is by definition bounded.  Consider the corresponding finite set from Lemma \ref{finite m vectors} and the resulting wall-and-chamber decomposition.  By throwing out those pseudo-walls which do not actually correspond to subobjects of some $E\in \SS$, we arrive at the following useful wall-and-chamber decomposition:

\begin{Prop} \label{prop:chambers}
There exists a locally finite set of \emph{walls} (pseudo-walls corresponding to genuine subobjects of semistable objects with Mukai vector $v$) 
in $\Stab(X)$, depending only on $v$, with the following properties: 
\begin{enumerate}
\item When $\sigma$ varies within a chamber, the sets of $\sigma$-semistable and
$\sigma$-stable objects of class $v$ does not change.
\item When $\sigma$ lies on a single wall $W\subset \Stab(X)$,
then there is a $\sigma$-semistable object
that is unstable in one of the adjacent chambers, and semistable in the other adjacent chamber.
\item When we restrict to an intersection of finitely many walls $W_1, \dots, W_k$, we
obtain a wall-and-chamber decomposition on $W_1 \cap \dots \cap W_k$ with the same properties,
where the walls are given by the intersections $W \cap W_1 \cap \dots \cap W_k$ for any
of the walls $W \subset \Stab(X)$ with respect to $v$.
\end{enumerate}
\end{Prop}

If $v$ is primitive, then from the proof of \cite[Proposition 9.4]{Bri08} $\sigma$ lies on a wall if and only if there exists a
strictly $\sigma$-semistable object of class $v$. 
From the above constructions, the Jordan-H\"older filtration of $\sigma$-semistable objects
does not change when $\sigma$ varies within a chamber.

\begin{Def}\label{def:generic}
Let $v\in K_{\num}(X)$.
A stability condition is called \emph{generic} with respect to $v$ if it does not lie on a wall
in the sense of Proposition \ref{prop:chambers}.
\end{Def}

We will also need the following useful fact \cite[Lemma 2.5]{BaMa}:
\begin{Lem} \label{lem:Yukinobualgebraic}
Consider a stability condition $\sigma = (Z, \AA)$ with $Z(v) = -1$. Then there are algebraic
stability conditions $\sigma_i = (Z_i, \AA_i)$ for $i = 1, \dots, m$ nearby $\sigma$
with $Z_i(v) = -1$ such that:
\begin{enumerate}
\item For every $i$ the following statement holds: an object of class $v$ is $\sigma_i$-stable (or
$\sigma_i$-semistable) if and only if it is $\sigma$-stable (or $\sigma$-semistable, respectively).
\item The central charge $Z$ is in the convex hull of $\{Z_1, \dots, Z_n\}$.
\end{enumerate}
\end{Lem}

\section{Moduli stacks of semistable objects}\label{sec:ModuliStacks}

We would like to use the results of \cite{Tod08} to construct for each stability condition $\sigma\in\Stab^{\dagger}(Y)$ a moduli stack of $\sigma$-semistable objectswhich are Artin stacks of finite type over $\C$.

Fix a smooth projective surface $X$ (to be either $Y$ or $\Y$ as above).  Let $\MMM_X$ be the 2-functor $$\MMM_X\colon(\text{Sch}/\C)\rightarrow (\text{groupoids}),$$ which sends a $\C$-scheme $S$ to the groupoid $\MMM_X(S)$ whose objects consist of $\mathcal E\in \mathrm{D}_{S\text{-}\mathrm{perf}}(S\times X)$ satisfying $$\Ext^i(\mathcal E_s,\mathcal E_s)=0,\text{ for all }i<0\text{ and }s\in S.$$   Lieblich proved in \cite{Lie} the following theorem: 

\begin{Thm} The 2-functor $\MMM_X$ is an Artin stack of locally finite type over $\C$.
\end{Thm}

Fix $\sigma=(Z,\mathcal P)\in \Stab(X),\phi\in \R,$ and $v\in H_{\alg}^*(X,\Z)$.  Then any object $E\in\mathcal P(\phi)$ satisfies $$\text{Ext}^i(E,E)=0,\text{ for all }i<0.$$  Indeed $\text{Ext}^i(E,E)=\text{Hom}(E,E[i])$, and $E\in \mathcal P(\phi)$ implies $E[i]\in \mathcal P(\phi+i)$.  Since $i<0$, $\phi+i<\phi$, and from the definition of a stability condition, we must then have $\text{Hom}(E,E[i])=0$.   

\begin{Def} Define $M_{\sigma,X}(v,\phi)$ to be the set of $\sigma$-semistable objects of phase $\phi$ and Mukai vector $v$, and $$\MMM_{\sigma,X}(v,\phi)\subset \MMM_X,$$ to be the substack whose fiberwise objects are in $M_{\sigma,X}(v,\phi)$.  As $\phi$ is determined $\mod \Z$ by $v$ and $\sigma$, we will drop it from the notation and assume henceforth that it is in $(0,1]$.  
\end{Def}

\begin{Rem}\label{phi one}  By Lemma \ref{lem:Yukinobualgebraic} and Remark \ref{rmk:GroupAction} above, we may in fact assume that $\phi=1$, $Z(v)=-1$, and $\sigma$ is algebraic.  We will make explicit when we are assuming this.
\end{Rem}

Toda proved in \cite{Tod08} the following helpful result:

\begin{Lem}\label{main lemma} Assume $M_{\sigma,X}(v)$ is bounded and $\MMM_{\sigma,X}(v)\subset \MMM_X$ is an open substack.  Then $\MMM_{\sigma,X}(v)$ is an Artin stack of finite type over $\C$.
\end{Lem}

This is the essential ingredient we need to prove the main theorem of this section:

\begin{Thm}\label{main theorem} Let $Y$ be an Enriques surface.  For any $v\in \Hal(Y,\Z)$ and $\sigma\in\Stab^{\dagger}(Y)$, $\MMM_{\sigma,Y}(v)$ is an Artin stack of finite type over $\C$.
\end{Thm}

We will prove the theorem using the results of \cite[Section 4]{Tod08} for the K3 surface $\Y$.  For the rest of this section, for any $\sigma\in\Stab^{\dagger}(Y)$, we will denote by $\sigma'\in\Gamma_{\Y}\cap\Stab^{\dagger}(\Y)$ a stability condition such that $(\pi^*)^{-1}(\sigma')=\sigma$.  Now we can easily prove the openness of $\sigma$-stability on $Y$:

\begin{Prop}\label{open} For any $v\in\Hal(Y,\Z)$ and $\sigma\in\Stab^{\dagger}(Y)$, $\MMM_{\sigma,Y}(v)$ is an open substack of $\MMM_Y$.
\end{Prop}
\begin{proof} By \cite[Lemma 3.6]{Tod08} this reduces to proving that for any smooth quasi-projective variety $S$ and $\mathcal E\in\MMM_Y(S)$ such that the locus $$S^{\circ}=\{s\in S|\mathcal E_s \in M_{\sigma,Y}(v)\},$$ is not empty, there is an open subset $U$ of $S$ contained in $S^{\circ}$.  Of course,  $(1\times \pi)^*\mathcal E\in\MMM_{\Y}(S)$, and by definition of induced stability conditions, the corresponding set $S^{\circ}$ for $(1\times \pi)^*\mathcal E$ and $M_{\sigma',\Y}(\pi^*v)$ remains the same.  By \cite[Section 4]{Tod08} there is an open set $U$ of $S$ contained in $S^{\circ}$ so the result follows.
\end{proof}

According to Lemma \ref{main lemma}, we only have to prove the boundedness of $M_{\sigma,Y}(v)$.  Let us first recall the following fundamental fact from \cite[Proposition 2.5]{BM}:

\begin{Lem}  Let $Y$ be an Enriques surface and $\Y$ its K3 universal cover.  
\begin{enumerate}
\item Let $F\in\Db(\Y)$.  Then there is an object $E\in \Db(Y)$ such that $\pi^*E\cong F$ if and only if $\iota^*F\cong F$.
\item Let $E\in \Db(Y)$.  Then there is an object $F\in \Db(\Y)$ such that $\pi_*F\cong E$ if and only if $E\otimes \omega_Y\cong E$.
\end{enumerate}
\end{Lem}

To compare stable objects on $Y$ and $\tilde{Y}$ we first make the following observation:

\begin{Lem}\label{2 to 1} If $E,F\in M_{\sigma,Y}(v)$ are $\sigma$-stable and $\pi^*E\cong\pi^*F$, then $E\cong F$ or $E\cong F\otimes \omega_Y$.
\end{Lem}
\begin{proof}  Indeed, pushing forward implies that $$E\oplus (E\otimes \omega_Y)\cong\pi_*\pi^*E\cong\pi_*\pi^*F\cong F\oplus (F\otimes \omega_Y).$$  Taking $\text{Hom}$'s gives that either $$\text{Hom}(E,F)\neq 0\text{ or }\text{Hom}(E,F\otimes\omega_Y)\neq 0.$$  But since $E$ and $F$ ($F\otimes \omega_Y$ respectively) are both $\sigma$-stable of the same phase, any non-zero homomorphism must be an isomorphism.  
\end{proof}

We will often need to exclude one of these possibilities: 
\begin{Lem}\label{exclude} If $E$ is $\sigma$-stable of phase $\phi$, then $\pi^*E$ is $\sigma'$-stable of the same phase, unless $E\cong E\otimes\omega_Y$, in which case $\pi^*E\cong F\oplus\iota^*F$, with $F\ncong \iota^*F$ $\sigma'$-stable objects of phase $\phi$, and thus not stable.  Moreover, in this case $E\cong \pi_*(F)\cong \pi_*(\iota^*F)$.
\end{Lem}
\begin{proof} By definition $\pi^*E$ is $\sigma'$-semistable, so suppose that it is strictly semistable.  Let $F\subset \pi^*E$ be a proper nontrivial $\sigma'$-stable subobject of the same phase $\phi$.  If $F\cong \iota^*F$, then there is a proper nontrivial $\sigma$-stable object $E'\subset E$ of phase $\phi$, contradicting stability of $E$.  

Otherwise, $F\ncong \iota^*F$ and $\iota^*F\subset \pi^*E$ is also $\sigma'$-stable of phase $\phi$.  Consider the short exact sequence $$0\rightarrow F\cap \iota^*F\rightarrow F\oplus\iota^*F\rightarrow F+\iota^*F\rightarrow 0,$$ which gives $$2Z(F)=Z(F\oplus\iota^*F)=Z(F\cap\iota^*F)+Z(F+\iota^*F),$$ where we write $Z(-)=Z_{\sigma'}(-),\phi(-)=\phi_{\sigma'}(-)$ to be concise.  By the see-saw principle and semistability of $F\oplus \iota^*F$, we must have either $$\phi(F\cap \iota^*F)<\phi<\phi(F+\iota^*F),\text{ or }\phi(F\cap\iota^*F)=\phi=\phi(F+\iota^*F).$$  Since $F+\iota^*F\subset \pi^*E$, semistability implies that we must have equality everywhere.  But then $F\cap\iota^*F\subset F$ of the same phase, so either $$F\cap\iota^*F=0\text{ or }F,$$ by the stability of $F$.  We assumed $F\ncong \iota^*F$, so we must be in the first case.  Thus $F\oplus\iota^*F\cong F+\iota^*F$ is an $\iota^*$-invariant nontrivial subobject of $\pi^*E$ of phase $\phi$.  It must thus come from a nontrivial subobject $F'\subset E$ of phase $\phi$.  If $F'$ is proper, equivalently $F\oplus\iota^*F$ is proper, then this contradicts the stability of $E$.  Thus we must have $F\oplus\iota^*F\cong \pi^*E$.  Pushing forward gives that $$E\oplus E\otimes\omega_Y\cong \pi_*(F)\oplus\pi_*(\iota^*F)\cong\pi_*(F)^{\oplus 2}.$$  From this and adjunction we deduce that \begin{align*}\Hom(E,E)\oplus\Hom(E,E\otimes \omega_Y)&=\Hom(E,\pi_*(F))^{\oplus 2}=\Hom(\pi^*E,F)^{\oplus 2}\\&=\Hom(F,F)^{\oplus 2}\oplus\Hom(F,\iota^*F)^{\oplus 2}.\end{align*}  Since $F$ and $\iota^*F$ are non-isomorphic $\sigma'$-stable objects of the same phase, $\Hom(F,\iota^*F)=0$, while $\Hom(E,E)=\Hom(F,F)=\C$.  Thus $\Hom(E,E\otimes\omega_Y)=\C$ which implies $E\cong E\otimes\omega_Y$ by stability.  Similar considerations show that $E\cong \pi_*(F)$.

For the converse, we have by adjunction that \begin{align*}\text{Hom}(\pi^*E,\pi^*E)&\cong\text{Hom}(E,\pi_*\pi^*E)\\&\cong\text{Hom}(E,E\oplus(E\otimes\omega_Y))\cong\text{Hom}(E,E)^{\oplus 2}.\end{align*}  Thus $\C^{\oplus 2}=\text{Hom}(\pi^*E,\pi^*E)$ implies $\pi^*E$ cannot be stable.
\end{proof}

To deduce the boundedness of $M_{\sigma,Y}(v)$ we again would like to compare this set with the corresponding set on $\Y$.  We start with a few quick observations:

\begin{Lem}\label{m vector} If $E$ is $\sigma$-stable, then $v(E)^2\geq -1$, unless $v(E)^2=-2$ which occurs precisely when $E$ is spherical.  
\end{Lem}
\begin{proof} From Serre duality and the definition of the Mukai pairing $$v(E)^2=\ext^1(E,E)-\hom(E,E)-\ext^2(E,E)=\ext^1(E,E)-\hom(E,E)-\hom(E,E\otimes\omega_Y).$$  By stability $\hom(E,E)=1$ and $\hom(E,E\otimes\omega_Y)=0$ or 1.  In the first case, $$v(E)^2+1=\ext^1(E,E)\geq 0.$$  In the latter case, $$v(E)^2+2=\ext^1(E,E)\geq 0,$$ and $v(E)^2<-1$ implies $v(E)^2=-2$, so $E$ is spherical as $\ext^2(E,E)=1$ implies $E\cong E\otimes\omega_Y$.
\end{proof}

\begin{Prop}\label{stable bounded}  Denote by $M^s_{\sigma,Y}(v)\subset M_{\sigma,Y}(v)$ the subset of $\sigma$-stable objects.  Then $M^s_{\sigma,Y}(v)$ is bounded.
\end{Prop}
\begin{proof} From Lemma \ref{exclude} we know that for any $E\in M^s_{\sigma,Y}(v)$ $\pi^*E$ is $\sigma'$-stable unless $E\cong E\otimes\omega_Y$ in which case $\pi^*E\cong F\oplus \iota^*F$ for $\sigma'$-stable objects $F\ncong\iota^*F$ of the same phase.  

Let us the consider the first case.  Then by boundedness of $M_{\sigma',\Y}(\pi^*v,\phi)$ \cite[Theorem 14.2]{Tod08}, there exists a scheme $Q$ of finite type over $\C$ and $\mathcal F\in \Db(Q\times\Y)$ such that every $F\in M_{\sigma',\Y}(\pi^*v,\phi)$ is equal to $\FF_q$ for some closed point $q\in Q$.  Consider the locally closed subscheme $$T:=\{q\in Q|\iota^*\mathcal F_q\cong \mathcal F_q,\FF_q\in M^s_{\sigma',\Y}(\pi^*v)\},$$ which is still of finite type over $\C$, and the restriction $\mathcal F_T$.  Then from \cite[Proposition 2.5]{BM} it follows that there exists $\mathcal E\in \Db(T\times Y)$ such that $(1\times\pi)^*(\mathcal E)\cong \mathcal F_T$.  Consider the disjoint union of two copies of $T$, which is still of finite type over $\C$, with $\mathcal E$ on the first copy of $T$ and $\mathcal E\otimes p_Y^*\omega_Y$ on the second.  Then by Lemma \ref{2 to 1} and the definition of induced stability conditions, it follows that $M^s_{\sigma,Y}(v)$ is bounded.

In the second case, consider $u\in\Hal(\Y,\Z)$ such that $\pi_*(u)=v$.  Note that by Lemma \ref{finite m vectors} only finitely many of these Mukai vectors appear as $v(F)$ for decompositions $E\cong F\oplus \iota^*F$. Then by boundedness of $M_{(\pi_*)^{-1}(\sigma),\Y}(u)$, we have a scheme $W$ of finite type over $\C$ and $\GG\in \Db(W\times \Y)$ representing every element $M_{(\pi_*)^{-1}(\sigma),\Y}(u)$.  Now consider the open set $$V:=\{w\in W|\iota^*\GG_w\ncong \GG_w,\GG_w\text{ is stable}\}$$ and $(1\times \pi)_*(\GG|_V)\in \Db(V\times Y)$.  Then $V$ is still of finite type.  Taking the finite union over the relevant $u$'s represents every member of $M^s_{\sigma,Y}(v)$.

Together these prove the claim.
\end{proof}

To prove boundedness in general, let us recall the following simple result:
\begin{Lem}[{\cite[Lemma 3.16]{Tod08}}]\label{bounded ext} Let $X$ be a smooth projective variety and subsets $\SS_i\subset \Db(X)$, $1\leq i\leq 3$, with $\SS_i$ bounded for $i=1,2$.  Suppose that any $E_3\in \SS_3$ sits in a distinguished triangle, $$E_1\rightarrow E_3\rightarrow E_2,$$ with $E_i\in \SS_i$ for $i=1,2$.  Then $\SS_3$ is also bounded.
\end{Lem}

\begin{Prop} $M_{\sigma,Y}(v)$ is bounded for any $\sigma$ and $v$.
\end{Prop}
\begin{proof}  By Lemma \ref{finite m vectors}, the number of Mukai vectors of possible stable factors for $E\in M_{\sigma,Y}(v)$ is finite.  By induction on the number of stable factors, we see that the claim follows from Proposition \ref{stable bounded} and Lemma \ref{bounded ext} above.
\end{proof}

As above we denote by $\MMM^s_{\sigma,Y}(v)\subset\MMM_{\sigma,Y}(v)$ the open substack parametrizing stable objects (and analogously for the corresponding sets of objects).  Inaba proved in \cite{Ina02} that $\MMM^s_{\sigma,Y}(v)$ is a $\mathbb G_m$-gerbe over a separated algebraic space that we denote by $M_{\sigma,Y}^s(v)$.  We have the following further result:

\begin{Lem} Fix $v\in\Hal(Y,	\Z)$.
\begin{enumerate}
\item The moduli stack $\MMM_{\sigma,Y}(v)$ satisfies the valuative criterion of properness.

\item Assume that $\MMM_{\sigma,Y}(v)=\MMM^s_{\sigma,Y}(v)$.  Then the course moduli space $M_{\sigma,Y}(v)$ is a proper algebraic space.
\end{enumerate}
\end{Lem}
\begin{proof} We follow \cite[Lemma 6.6]{BaMa}.  Then we can use the $\widetilde{\text{GL}}_2^+(\R)$ action to assume that $\phi=1$ and $Z(v)=-1$.  We may assume that $\sigma$ is algebraic and thus that $\mathcal P(1)$ is Noetherian.  But then \cite[Theorem 4.1.1]{AP06} implies the lemma.
\end{proof}

We finish this section by proving the non-emptiness of these moduli spaces.  To do so we must first recall a result of Bridgeland comparing Bridgeland stability and Gieseker stability in the large radius limit.

\begin{Thm}[{\cite[Proposition 14.1]{Bri08}}]\label{BG comp} Let $v\in\Hal(\Y,\Z),\beta\in\NS(\Y)_{\Q},$ and $H\in\Amp(\Y)$ with $\mu_{H,\beta}(v)>0$, and set $\omega=tH$.  Then $M_{\sigma'_{\omega,\beta},\Y}(v)=M_{H,\Y}^{\beta}(v)$ for $t\gg0$.
\end{Thm}

A final preparatory step before proving non-emptiness is a result which is of great interest in its own right.  While unnodal Enriques surfaces admit no spherical objects \cite[Lemma 3.17]{MMS} and thus no spherical (Seidel-Thomas) twists, they do have closely related derived auto-equivalences corresponding to \emph{weakly-spherical} objects.  These are objects $E\in\Db(Y)$ with $\ext^i(E,E)=0$ for $i\neq 0$ and $\hom(E,E)=1$, so in particular $E\ncong E\otimes \omega_Y$ and $v(E)^2=-1$.  It follows immediately that $\pi^*E$ is a spherical object on $\tilde{Y}$.  We have the following result about the associated spherical twist $\ST_{\pi^*E}(-)$:

\begin{Prop} Let $E\in\Db(Y)$ be a weakly-spherical object and $\ST_{\pi^*E}(-)$ the spherical twist associated to $\pi^*E$, i.e. the derived auto-equivalence defined by the exact triangle $$\Hom^{\bullet}(\pi^*E,F)\otimes \pi^*E\to F\to \ST_{\pi^*E}(F)$$ for every $F\in\Db(Y)$.  Then $\ST_{\pi^*E}$ preserves $\Db(\tilde{Y})_G\cong\Db(Y)$ and thus descends to an auto-equivalence on $\Db(Y)$.  The effect on cohomology is the map $$v(F)\mapsto v+2(v(F),v(E))v(E).$$
\end{Prop}
\begin{proof}  The important observation here is that both $\pi^*E$ and $F\in \Db(\tilde{Y})_G$, so $\Hom^{\bullet}(\pi^*E,F)$ is $G$-invariant and thus the first morphism is in $\Db(\tilde{Y})_G$.  Completing it to an exact triangle stays inside $\Db(\tilde{Y})_G$, and thus we see that $\ST_{\pi^*E}(F)$ is $G$-invariant as well.  It follows that $\ST_{\pi^*E}$ descends to an auto-equivalence on $\Db(Y)$.

The statement about the action on cohomology follows from the above description and \cite[Lemma 8.12]{Hu1}.
\end{proof}

These auto-equivalences were referred to as Fourier-Mukai transforms associated to $(-1)$-reflection in \cite{Yos03}, but the above interpretation strengthens and elucidates the connection with the covering spherical twist on the K3 surface $\tilde{Y}$.  For brevity we call these \emph{weakly-spherical twists}.  Now we are ready for the proof of non-emptiness.

\begin{Thm}\label{nonempty} Let $v=mv_0$ with $m\in\Z_{>0}$ and $v_0=(r,C,s)$ a primitive vector with $v_0^2\geq -1$.  Suppose that $M_{H',Y}(w)\neq \varnothing$ for all primitive and positive $w\in\Hal(Y,\Z)$ (obtainable from $v_0$ by applying weakly-spherical twists) and polarization $H'$ generic with respect to $w$,  then $\MMM_{\sigma,Y}(v)(\C)\neq\varnothing$ for all $\sigma\in \Stab^{\dagger}(Y)$.
\end{Thm}
\begin{proof} Since we are interested at the moment in semistable objects, it suffices to consider the case when $m=1$, i.e. $v$ is primitive.  Indeed, if $E_0\in \MMM_{\sigma,Y}(v_0)(\C)$, then $E=E_0^{\oplus m}\in \MMM_{\sigma,Y}(v)(\C)$.  By the remarks preceeding \cite[Lemma 8.2]{Bri07}, semistability is a closed condition, so it also suffices to suppose that $\sigma$ is generic with respect to $v$ so that every $\sigma$-semistable object of class $v$ is stable since these remain at least semistable at the boundary.

Now the construction of the Joyce invariant $J(v)$ of \cite[Section 5]{Tod08} is quite general, and Lemma 5.12 there applies.  Likewise the analogous algebra $A(\mathcal A_{\phi},\Lambda,\chi)$ is still commutative since $\omega_Y$ is numerically trivial and thus the Mukai pairing is commutative.  This and the results above show that \cite[Theorem 5.24 and Corollary 5.26]{Tod08} still apply.  In particular, $J(v)$ is the motivic invariant of the proper coarse moduli space $M_{\sigma,Y}(v)$, and is invariant under autoequivalences and changes in $\sigma$.  We can thus assume that $v_0$ is positive.  Indeed, if $r\neq 0$, then we can shift by 1, i.e. $E\mapsto E[1]$, to make $r>0$ if necessary.  If $r=0$ but $s\neq 0$, then we can apply the weakly-spherical twist through $\OO_Y$ and a shift, if necessary, to make $v$ positive.  Finally, we are reduced to the case $v=(0,C,0)$.  We can tensor with $\OO(D)$ for any $D\in\Pic(Y)$ such that $D.C\neq 0$ (we can even choose $D$ such that $D.C=1$ since $\Pic(Y)$ is unimodular), and then apply the weakly-spherical twist through $\OO_Y$ and a shift to make $v$ positive.

Now let $H=\pi^*H'$ where $H'\in \Amp(Y)$ is generic with respect to $v$ such that $\mu_{H,0}(\pi^*v)>0$.  Theorem \ref{BG comp} shows that for $\omega=tH$ and $t\gg 0$, being $\sigma'_{\omega,0}$-semistable of class $\pi^*v$ is equivalent to being Gieseker semistable.  Let $\omega'=tH'$ and set $\sigma_{\omega',0}=(\pi^*)^{-1}(\sigma'_{\omega,0})$, so that $E\in M_{\sigma_{\omega',0},Y}(v)$ if and only if $\pi^*E$ is $H$-Gieseker semistable.  Of course, it thus follows that $E$ is $H'$ Gieseker semistable on $Y$.  $H'$ being generic gaurantees that all Gieseker semistable objects are stable.  

Thus we may choose $\sigma$ such that the coarse moduli space $M_{\sigma,Y}(v)$ is the moduli space $M_{H',Y}(v)$ of Gieseker stable sheaves on $Y$ with Mukai vector $v$ for a generic polarization $H'$.    Then $J(v)$ is the motivic invariant of $M_{H',Y}(v)$ which is non-trivial by assumption, so it follows that $\MMM_{\sigma,Y}(v)(\C)\neq\varnothing$ for all $\sigma$.
\end{proof}

\begin{Rem} By Theorems \ref{yosh odd} and \ref{hauz even} above, the hypothesis of the theorem is satisfied for $Y$ unnodal.  Moreover, in this case there are no spherical objects \cite[Proposition 3.17]{MMS}, so we need not consider the exceptional case $v_0^2=-2$.  Thus all relevant cases are covered by the above theorem.
\end{Rem}

\section{The Geometry of the Morphism $\Phi$}

We begin here our investigation of the relationship between the geometry of the moduli spaces $\MMM_{\sigma,Y}(v)$ and $\MMM_{\sigma',\Y}(\pi^*v)$, where again $\sigma=(\pi^*)^{-1}(\sigma')$ for some invariant $\sigma'$.  Notice that $\pi^*$ induces a morphism of stacks $$\Phi:\MMM_{\sigma,Y}(v,\phi)\rightarrow\MMM_{\sigma',\Y}(\pi^*v,\phi).$$  Since $\iota$ induces an autoequivalence of $\Db(\Y)$, and we've chosen $\sigma'\in\Gamma_{\Y}$, $\iota$ induces an involution on $\MMM_{\sigma',\Y}(\pi^*v)$.  It follows that $\Phi$ factors through the fixed point substack $\Fix$, a closed substack, to give a morphism $$\Phi:\MMM_{\sigma,Y}(v)\rightarrow \Fix,$$ which we still denote by $\Phi$\footnote{To avoid too many stack-theoretic complications, we will define the fixed point stack as $\Fix:=M_{\sigma',\Y}(\pi^*v)^G\times_{M_{\sigma',\Y}(\pi^*v)}\MMM_{\sigma',\Y}(\pi^*v)$, where $M_{\sigma',\Y}(\pi^*v)^G\subset M_{\sigma',\Y}(\pi^*v)$ is the fixed-point subscheme of the coarse moduli space.  This ensures that the fixed point substack is in fact a closed substack, smooth if the ambient stack is.  It also ensures that the morphism to the fixed point substack descends through Inaba's rigidification by $\G_m$.  For a more general and thorough discussion of these issues, see \cite{Rom}}.   

As usual, we start by considering the stable locus and generalize the results and arguments of \cite{Kim},\cite{Sac} to the case of Bridgeland moduli spaces.  First we note the following fact:

\begin{Lem}\label{isotropic} $\Fix\cap M^s_{\sigma',\Y}(\pi^*v)$ is a union of isotropic algebraic subspaces of $M^s_{\sigma',\Y}(\pi^*v)$.
\end{Lem}
\begin{proof}  In \cite[Theorem 3.3]{Ina11} Inaba generalized the by-now classical result from \cite{Muk84} that the moduli space of stable sheaves on a K3 surface $\Y$ carries a non-degenerate symplectic form.  Recall that for $F\in M_{\sigma',\Y}(\pi^*v)$ Inaba defined the sympletic form $\omega$ on the smooth algebraic space $M^s_{\sigma',\Y}(\pi^*v)$ by considering the composition \begin{align*}\Ext^1(F,F)\times \Ext^1&(F,F) \rightarrow \Ext^2(F,F)\rightarrow H^2(\Y,\mathcal O_{\Y})=H^2(\Y,\omega_{\Y})\cong \C\\&(e,f)\longmapsto e\cup f\text{              }\longmapsto \text{tr}(e\cup f),\end{align*} where the identification of $H^2(\Y,\omega_{\Y})$ with $\C$ is dual to the isomorphism between $H^0(\Y,\omega_{\Y})$ and $\C$, where the former is generated by the unique holomorphic 2-form $\alpha$ up to scaling.  Since $\iota^*$ sends $\alpha$ to $-\alpha$, it follows that $\omega$ is anti-sympletic, i.e. $\omega(\iota^*e,\iota^*f)=-\omega(e,f)$.  

Moreover, as $M^s_{\sigma',\Y}(\pi^*v)$ is a smooth algebraic space, $\Fix\cap M^s_{\sigma',\Y}(\pi^*v)$ is the union of smooth subspaces.  The fact that it is isotropic follows from the fact that $\iota^*$ is anti-symplectic.
\end{proof}

\begin{Prop}\label{lagrange} The morphism of stacks $$\Phi:\MMM_{\sigma,Y}(v)\rightarrow \Fix\subset \MMM_{\sigma',\Y}(\v)$$ is onto.  The induced morphism $${\Phi}^s:\MYs\rightarrow\Fix$$ is a 2-to-1 cover onto its image, \'{e}tale away from those points with $E\cong E\otimes\omega_Y$.
\end{Prop}
\begin{proof}  First we show that $\Phi$ is surjective onto $\Fix$.  Indeed, suppose $F\in \mX$ is invariant under $\iota^*$.  From \cite[Proposition 2.5]{BM}, there exists an object $E\in \Db(Y)$ such that $\pi^*E\cong F$.  From the definition of induced stability conditions, it follows that $E\in \mY$.  Moreover, it clearly follows that if $F$ is stable, then so is $E$.  

Now we prove that $\Phi^s$ is unramified.  Lieblich and Inaba (in \cite{Lie} and \cite{Ina02}, respectively) generalized the well-known results about the deformation theory of coherent sheaves to complexes of such.  In particular, for $E\in \mYs$ and $F=\pi^*E$, the tangent spaces are $$T_{E}\MY\cong \Ext^1(E,E),\text{ and }T_{F}\MX\cong \Ext^1(F,F),$$ and the differential is just the natural map $$d\Phi:\Ext^1(E,E)\rightarrow \Ext^1(F,F).$$  Note that it follows from Riemann-Roch that if $E$ and $F=\pi^*E$ are both stable, $\MY$ is smooth at $E$ (since the obstruction space vanishes because $\Ext^2(E,E)=0$) of dimension $\dim T_E=v^2+1$ while $\MX$ is smooth at $F$ of dimension $(\v)^2+2=2\dim \MY$.

We claim that the differential must be injective for $E\ncong E\otimes \omega_Y$.  Indeed, suppose $E'\in \text{Ext}^1(E,E)$, i.e. $E'$ is an extension $$0\rightarrow E\rightarrow E'\rightarrow E\rightarrow 0,$$ in $\mathcal P_{\sigma}(\phi)$.  Notice from applying $\Hom(-,E\otimes \omega_Y)$ and noting that $E$ and $E\otimes\omega_Y$ are nonisomorphic and stable of the same phase so that $\Hom(E,E\otimes \omega_Y)=0$, we must have $\Hom(E',E\otimes \omega_Y)=0$.  Suppose that $\pi^*E'=0\in\Ext^1(F,F)$, i.e. the short exact sequence $$0\rightarrow F\rightarrow \pi^*E'\rightarrow F\rightarrow 0$$ in $\mathcal P_{\sigma'}(\phi)$ splits.  But then so does the short exact sequence $$0\rightarrow \pi_*(F)\rightarrow \pi_*(\pi^*E')\rightarrow \pi_*(F)\rightarrow 0.$$  But this is precisely the sequence $$0\rightarrow E\oplus (E\otimes \omega_Y)\rightarrow E'\oplus (E'\otimes \omega_Y)\rightarrow E\oplus (E\otimes \omega_Y)\rightarrow 0.$$  Since $\Hom(E',E\otimes \omega_Y)=\Hom(E'\otimes \omega_Y,E)=0$, it follows that any morphism $$E'\oplus (E'\otimes \omega_Y)\rightarrow E\oplus(E\otimes \omega_Y)$$ must be component wise, and thus any splitting of this short exact sequence induces a splitting of the original exact sequence $$0\rightarrow E\rightarrow E'\rightarrow E\rightarrow 0,$$ proving injectivity.

Finally, note that $d\Phi$ factors through $T_F\Fix$.  Since $\Fix\cap \MXs$ is smooth and isotropic by Lemma \ref{isotropic}, it follows that $d\Phi$ is isomorphic onto $T_F\Fix$, so $\Phi$ is \'{e}tale at $E$.  That it is 2-to-1 follows from Lemma \ref{2 to 1}.
\end{proof}

\begin{Rem} If $\Fix\cap \MXs$ is nonempty, then it follows that every component is a Lagrangian substack from the above proposition.
\end{Rem}

\section{Projectivity of Coarse Moduli Spaces for Unnodal Enriques Surfaces}\label{sec:Projectivity}

Throughout this section we assume that $Y$ is unnodal.  In this case, we know that $\iota^*$ acts as the identity on $\Hal(\Y,\Z)$, $\Stab^{\dagger}(\Y)$ is mapped isomorphically to $\Sigma(Y)$ so that $\Stab^{\dagger}(\Y)=\Sigma(Y)=\Stab^{\dagger}(Y)$, and $\Db(Y)$ contains no spherical objects (see \cite[Lemma 3.10, Proposition 3.12, and Lemma 3.14]{MMS}).  We emphasize the most important consequence of this assumption with the following observation:
\begin{Obs}\label{gen un} A wall in the wall-and-chamber decomposition of $\Stab^{\dagger}(Y)$ corresponding to $v$ is still a wall in $\Stab^{\dagger}(\Y)=\Stab^{\dagger}(Y)$ corresponding to $\v$, though not necessarily conversely.  As such, we may choose $\sigma=\sigma'$ to be generic with respect to both wall-and-chamber decompositions.  We assume this for the remainder of the paper unless we explicitly drop the assumption that $Y$ is unnodal.
\end{Obs}

Let us begin this section with the following corollary of all of the work above:
\begin{Cor}\label{primitive} Suppose $Y$ is unnodal, $v\in\Hal(Y,\Z)$ is primitive with $v^2\geq -1$, and  $\sigma\in\Stab^{\dagger}(Y)$ generic with respect to $v$.  Then there is a coarse moduli space $\mY$ which is a non-empty projective variety parametrizing only stable objects.  It has dimension $v^2+1$ unless $v^2=0$ and $\v$ is divisible by 2, in which case it has dimension $2=v^2+2$.
\end{Cor}
\begin{proof} Since $v$ is primitive and $\sigma$ generic, $\mYs=\mY$ is a proper algebraic space, nonempty as $Y$ is unnodal.  Proposition \ref{lagrange} shows that $\mY$ is a finite cover of $\Fix$, a (smooth) closed subvariety of the projective coarse moduli space $\mX$ (which exists and is projective by \cite[Theorem 1.3]{BaMa}).  Thus $\mY$ is projective as well.  

For the statement about dimension, note that by standard deformation theory arguments, $$v^2+1\leq \dim_E \mY\leq \dim T_E \mY=v^2+1+\hom(E,E\otimes\omega_Y)\leq v^2+2,$$ for any $E\in\mYs$.  Singular points must then satisfy $E\cong E\otimes\omega_Y$.  But then $\pi^*E=F\oplus \iota^*F$ as in Lemma \ref{exclude}, so  $$\v=v(\pi^*E)=v(F)+v(\iota^*F)=v(F)+\iota^*v(F)=2v(F).$$  Thus this is only possible if $\v$ is divisible by 2.  Now since $\Phi$ is \'{e}tale away from the fixed locus of $-\otimes\omega_Y$ and $\Fix$ is a Lagrangian submanifold, we get that $\mY$ has dimension $v^2+1$ unless it consists entirely of objects fixed by $-\otimes\omega_Y$ and has dimension $v^2+2$.  We will see in Theorem \ref{singular} that this is only possible if $v^2=0$ and $\v$ is divisible by 2.
\end{proof}
\begin{Rem} The above proof still works for nodal $Y$ as long as $\MY(\C)\neq\varnothing$ and $\sigma=(\pi^*)^{-1}(\sigma')$ can be chosen such that $\sigma'$ is generic with respect to $\v$ or projective coarse moduli spaces for non-generic $\sigma'$ on K3 surfaces are constructed.
\end{Rem}

Armed with the above preparation, we can begin to tackle the non-primitive case.  Let $v=mv_0\in\Hal(Y,\Z)$ with $v_0$ primitive and $m\in\Z_{>0}$.  By Lemma \ref{m vector} we must assume $v_0^2\geq -1$ for $Y$ unnodal.  As a warm-up we begin with the cases with $v_0^2\leq 0$.

\begin{Lem}\label{point} Assume $v_0^2=-1$.  Then for all $\sigma\in\Stab^{\dagger}(Y)$ generic with respect to $v$, $\MY$ admits a projective coarse moduli space $\mY$ consisting of $m+1$ points.
\end{Lem}
\begin{proof}  The proof of \cite[Lemma 7.1]{BaMa} shows that for $m=1$ the stack $\MX$ is a $\mathbb G_m$-gerbe over a point representing a single object $F_0$, which must be spherical and fixed by $\iota^*$.  Thus it descends to an object $E_0$ in $\mYs=\mY$.  By Lemma \ref{exclude} we must have $\Hom(E_0,E_0\otimes\omega_Y)^{\vee}=\Ext^2(E_0,E_0)=0$ since $F_0$ is stable.  Thus $\MY$ is smooth at $E_0$.  It follows that $\mY$ consists of two reduced points $E_0$ and $E_0\otimes \omega_Y$.  

If $m>1$, then the argument in \cite[Lemma 7.1]{BaMa} shows that every $\sigma'$-semistable object with Mukai vector $\v$ must be of the form $F_0^{\oplus m}$.  We notice that $$\ext^1(E_0,E_0)=\ext^1(E_0\otimes\omega_Y,E_0\otimes\omega_Y)=\ext^1(E_0,E_0\otimes\omega_Y)=\ext^1(E_0\otimes\omega_Y,E_0)=0.$$  Indeed  \begin{align*}-1=v_0^2&=(v(E_0),v(E_0\otimes\omega_Y))\\&=\ext^1(E_0,E_0\otimes\omega_Y)-\hom(E_0,E_0\otimes\omega_Y)-\hom(E_0,E_0)\\&=\ext^1(E_0,E_0\otimes\omega_Y)-1.\end{align*}  By genericity of $\sigma$, all stable factors of an element of $\mY$ must have Mukai vector $m'v_0$ for $m'<m$, so by induction we conclude that the only $\sigma$-semistable objects with Mukai vector $v$ are precisely $E_0^{\oplus m},E_0^{\oplus m-1}\oplus (E_0\otimes \omega_Y),...,E_0\oplus (E_0\otimes\omega_Y)^{\oplus m-1},(E_0\otimes\omega_Y)^{\oplus m}.$
\end{proof}

\begin{Lem}\label{isotrop} Assume that $v_0^2=0$.  Let $\sigma$ be generic with respect to $v$.  Then:
\begin{enumerate}
\item for $m=1$, $\mY$ is an irreducible smooth projective curve if $\v$ is primitive or isomorphic to $Y$ itself if $\v$ is divisible by 2.

\item for $m>1$, 
\begin{itemize}
\item if $\v$ is primitive, then a projective coarse moduli space $\mY$ exists and $$\mY\cong\coprod_{2m_1+m_2=m} \Sym^{m_1}(M^s_{\sigma,Y}(2v_0))\times \Sym^{m_2}(M_{\sigma,Y}(v_0)).$$
\item if $\v$ is divisible by 2, then a projective coarse moduli space $\mY$ exists and $$\mY\cong \Sym^m(M_{\sigma,Y}(v_0)).$$
\end{itemize}
\end{enumerate}
\end{Lem}
\begin{proof}  Corollary \ref{primitive} shows that for $m=1$, $\mY=\mYs$ is a non-empty smooth projective variety of dimension 1 if $\v$ is primitive.  If $\v$ is divisible by 2, then $\frac{1}{2}\v$ is primitive and isotropic, so $M_{\sigma',\tilde{Y}}(\frac{1}{2}\v)$ is a smooth projective K3 surface parametrizing stable objects by \cite[Lemma 7.2(a)]{BaMa}.  Then $$\pi_*:M_{\sigma',\tilde{Y}}(\frac{1}{2}\v)\to \mY,$$  gives a 2-to-1 covering, which can easily see to be \'{e}tale.  Indeed, only $F$ and $\iota^*F$ are sent to the same object $E$, which is necessarily stable, so there cannot be any object of $M_{\sigma',\tilde{Y}}(\frac{1}{2}\v)$ fixed by $\iota^*$.  This gives a component of $\mY$ which is the quotient of a projective K3 surface by a fixed-point free involution, i.e. an Enriques surface.  Irreducibility in both cases and the fact that $\mY$ is in fact isomorphic to $Y$ itself in case $\v$ is divisible by 2 follow exactly as in Section 8 of \cite{Nue14a}.  This proves part (a).

 For the proof of (b), first consider the case recall that $\mX\cong\Sym^m(M_{\sigma',\Y}(\pi^*v_0))$ from \cite[Lemma 7.2(b)]{BaMa}.  It follows that the stable locus $\mXs=\varnothing$ since on the one hand it would be a dense open subset of $\mX$ which has dimension $2m$, and on the other hand it would also have to be smooth of dimension $(\pi^*mv_0)^2+2=2$, which is impossible for $m>1$.  Thus for any $E\in\mYs$, we would have to be in the exceptional case of Lemma \ref{exclude}, so $E\cong E\otimes\omega_Y$ (and the same is true on the entire component containing $E$) and $\pi^*E\cong F\oplus \iota^*F$ where $F$ is stable of Mukai vector $\frac{m}{2}\pi^*v_0$ and $F\ncong \iota^*F$.  As noted above, for stable objects to exist on $\Y$ we must have $m/2=1$, i.e. $m=2$.  As for the semistable locus, it follows from the genericity of $\sigma$ that any stable factor of a semistable object $E\in \mY$ must have Mukai vector $m'v_0$ for $m'<m$.  Repeating the above argument inductively, we find that a canonical representative of the $S$-equivalence class of an object $E\in \mY$ is a direct sum of objects in $M^s_{\sigma,Y}(2v_0)$ and objects in $M_{\sigma,Y}(v_0)$.  Thus the coarse moduli space parametrizing $S$-equivalence classes is $$\coprod_{2m_1+m_2=m} \Sym^{m_1}(M^s_{\sigma,Y}(2v_0))\times \Sym^{m_2}(M_{\sigma,Y}(v_0)).$$  

Since the morphism $\pi^*$ from $\mY$ to $\mX$ is quasi-finite (as follows from the above decomposition) and proper (as the two Artin stacks themselves were proper), we find that the morphism of coarse moduli spaces is finite.  Thus $\mY$ is projective.

Now consider the second case in (b).  As usual, from the genericity of $\sigma$ it follows that any stable factors of an object in $\mY$ must be of the form $m'v_0$ for $m'<m$.  By the arguments above, $\mYs=\varnothing$ for $m>1$, so it follows that the $S$-equivalence classes of the objects in $\mY$ are represented by $\Sym^m(M_{\sigma,Y}(v_0))$.  
\end{proof}

We can now generalize the above argument to show projectivity in general:
\begin{Thm}\label{projective general} Let $v=mv_0$, $m>0$, be a Mukai vector with $v_0$ primitive and $v_0^2>0$.  Then a projective coarse moduli space $\mY$ exists.
\end{Thm}
\begin{proof} As usual we notice that the genericity of $\sigma$ means that any stable factor of an object of $\mY$ must have Mukai vector $m'v_0$ for $m'<m$, which implies that the strictly semistable locus is the image of the natural map $$\SSL: \coprod_{m_1+m_2=m,m_i>0} M_{\sigma,Y}(m_1v_0)\times M_{\sigma,Y}(m_2v_0)\rightarrow \mY.$$  Now if for two strictly semistable objects $E$ and $E'$, $\pi^*E$ and $\pi^*E'$ are $S$-equivalent, then the stable factors coincide and appear with the same multiplicities in the graded object.  But a stable factor of $E$ (or $E'$) remains stable after pull-back unless it is fixed under $-\otimes\omega_Y$.  For such a stable factor $S$ we have stable $Q\in \Db(\Y)$ such that $S\cong \pi_*(Q)$, or equivalently $\pi^*S=Q\oplus \iota^*Q$ with $Q\ncong \iota^*Q$.  All of this implies that $E$ and $E'$ have the same stable factors that are fixed under $-\otimes\omega_Y$, with the same multiplicites, and their stable factors that are not invariant under $-\otimes\omega_Y$ can only differ by tensoring with $\omega_Y$.  Thus again the proper morphism $$\pi^*:\mY\rightarrow\mX$$ between coarse moduli spaces is quasi-finite and thus finite.  Since the latter is projective by \cite[Theorem 1.3]{BaMa}, $\mY$ must be projective as well.  
\end{proof}

\begin{table}[ht]
\caption{Dimension of moduli spaces and their semistable loci}
\begin{tabular}{c c c c c c}
\hline
$v_0^2$ & $\v_0$ & $m$ & $\mYs$ & $\dim \mY$ & $\codim \mYss$ \\
\hline
-1 & $-$ & 1& $\neq\varnothing$ & $0=v^2+1$ & $\infty$\\
-1 & $-$ & $>1$ & $\varnothing$ & $0\neq v^2+1$ & 0\\
0 & primitive & 1 & $\neq\varnothing$ & $1=v^2+1$ & $\infty$\\
0 & primitive & 2 & $\neq\varnothing$ & $2\neq v^2+1$ & 0\\
0 & primitive & $>2$ & $\varnothing$ & $m\neq v^2+1$ & 0\\
0 & non-primitive & 1 & $\neq\varnothing$ & $2\neq v^2+1$ & $\infty$\\
0 & non-primitive & $>1$ & $\varnothing$ & $2m\neq v^2+1$ & 0\\
1 & $-$ & $1,>2$ & $\neq\varnothing$ & $v^2+1$ & $\infty,>1$\\
1 & $-$ & 2 & $\neq\varnothing$ & $v^2+1$ & 1\\
$>1$ & $-$ & $m\geq 1$ & $\neq\varnothing$ & $v^2+1$ & $>1$\\

\hline
\end{tabular}
\label{table:dimension}
\end{table}

Inspired by \cite[Theorem 2.15]{BaMa13}, we can use the above technique to determine the dimension of $\mY$ and of its semistable locus, as well as to ensure the existence of stable objects.
\begin{Thm}\label{num dim} Let $v=mv_0$ be a Mukai vector with $v_0$ primitive and $m>0$ with $\sigma\in\Stab^{\dagger}(Y)$ generic with respect to $v$.  
\begin{enumerate}
\item The coarse moduli space $\mY\neq \varnothing$ if and only if $v_0^2\geq -1$.
\item The dimension and codimension of $\mY$ and $\mYss$, respectively, follow Table \ref{table:dimension}

\end{enumerate}
\end{Thm}
\begin{proof} If $v_0^2\geq -1$, then part (a) follows from Theorem \ref{nonempty} above.  For the converse, note that any stable factor of an element of $\mY\neq \varnothing$ would have to have Mukai vector $m'v_0$ for $m'<m$ by genericity of $\sigma$.  But then $m'^2v_0^2=(m'v_0)^2\geq -1$, so $v_0^2\geq -1$.

For (b), again notice that the genericity of $\sigma$ means that any stable factor of an object of $\mY$ must have Mukai vector $m'v_0$ for $m'<m$, which implies that the strictly semistable locus is the image of the natural map 
$$\SSL: \coprod_{m_1+m_2=m,m_i>0} M_{\sigma,Y}(m_1v_0)\times M_{\sigma,Y}(m_2v_0)\rightarrow \mY.$$  

Assume $v_0^2>0$.  Then for $m=1$, $\mY=\mYs$, and we have seen already that $\dim \mY=v^2+1$.  If $m>1$, then by the induction, we deduce that the image of the map $\SSL$ has dimension equal to the maximum of $(m_1^2+m_2^2)v_0^2+2$ for $m_1+m_2=m,m_i>0$.  

We can construct a semistable object $E'$ with Mukai vector $v$ which is also Schur, i.e. $\Hom(E',E')=\C$.  By the inductive assumption, we can consider $E\in M_{Y,\sigma}^s((m-1)v_0)$, and let $F\in M_{Y,\sigma}(v_0)$.    Now $\chi(F,E)=-(v(F),v(E))=-(m-1)v_0^2<0$, so $\text{Ext}^1(F,E)\neq 0$.  Take $E'$ to be a nontrivial extension $$0\rightarrow E\rightarrow E'\rightarrow F\rightarrow 0.$$  Then any endomorphism of $E'$ gives rise to a homomorphism $E\rightarrow F$, of which there are none since these are both stable of the same phase and have different Mukai vectors (or can be chosen to be non-isomorphic if $m=2$).  Thus any endomorphism of $E'$ induces an endomorphism of $E$, and the kernel of this induced map $\Hom(E',E')\rightarrow \Hom(E,E)=\C$ is precisely $\Hom(F,E')$, which vanishes since the extension is non-trivial.  Thus $\Hom(E',E')=\C$.

We can deduce non-emptiness of $\mYs$ from a dimension estimate as follows.  Since $E'$ is Schur, we get $$v^2+1\leq \dim_{E'} \mY\leq \dim T_{E'}\mY=v^2+1+\hom(E',E'\otimes\omega_Y).$$  Notice that the strictly semistable locus must have dimension smaller than $v^2+1$.  So even though $E'$ is not stable, it lies on a component which must contain stable objects.  Moreover, as we will see in the next section, (smooth) components of the stable locus of dimension greater than $v^2+1$ can occur only if $v_0^2=0$, so in the current situation the locus of points fixed by $-\otimes\omega_Y$ has positive codimension.  Then we may choose $E\in M_{Y,\sigma}^s((m-1)v_0)$ such that $E\ncong E\otimes\omega_Y$ and $F$ such that $F\ncong F\otimes\omega_Y$ (and such that $F\ncong E\otimes\omega_Y$ if $m=2$).  Stability of $E$ and $F$ and a diagram chase then show that $\Hom(E',E'\otimes\omega_Y)=0$, so $\mY$ is smooth at $E'$ of dimension $v^2+1$ as claimed.

Furthermore, observe that the strictly semistable locus has codimension $$v^2+1-(m_1^2v_0^2+m_2^2v_0^2+2)=(m_1+m_2)^2v_0^2+1-(m_1^2v_0^2+m_2^2v_0^2+2)=2m_1m_2v_0^2-1\geq 2,$$ if $v_0^2>1$ or $m>2$, hence part (c).

The cases with $v_0^2\leq 0$ have already been covered in Lemmas \ref{point} and \ref{isotrop}.
\end{proof}

\section{Singularities of Bridgeland Moduli Spaces and Kodaira dimension}\label{sec:SingKod}

We've seen above that for $Y$ an unnodal Enriques surface, $v$ a primitive Mukai vector such that $\v$ is primitive as well, and $\sigma\in\Stab^{\dagger}(Y)$ generic with respect to $v$, the coarse moduli space $\mY$ is a smooth projective variety of dimension $v^2+1$ representing only stable objects.  It follows that singularities can occur only for $Y$ nodal, $v$ not primitive, or $v$ primitive with $\v$ divisible by 2.  For the remainder of this section we drop the assumption that $Y$ is unnodal and instead assume that one has shown that a non-empty coarse moduli space exists.  Assuming this we can nevertheless describe quite well the structure these moduli spaces must have, generalizing the results of \cite{Kim} and \cite{Yamada}.  As usual we denote by $\Y$ the K3 cover and $\sigma=(\pi^*)^{-1}(\sigma')$.

The main theorem of \cite{Kim} generalizes to Bridgeland moduli spaces without much change, but we present it here for the sake of completeness:
\begin{Thm}\label{singular} Let $Y$ be an Enriques surface, $v\in\Hal(Y,\Z)$, and $\sigma\in\Stab^{\dagger}(Y)$ (not necessarily generic).  Then the algebraic space $\mYs$ is singular at $E$ if and only if $E\cong E\otimes\omega_Y$ and $E$ lies on a component of dimension $v^2+1$.  The singular locus of $\mYs$ is the union of the images under $\pi_*$ of finitely many components of the algebraic spaces $$M^s_{\sigma',\Y}(w)^{\circ}=\{F\in M^s_{\sigma',\Y}(w)|F\ncong\iota^*F\},$$ as $w\in\Hal(\Y,\Z)$ ranges over classes such that $\pi_*(w)=v$.   Consequently, $$\dim \Sing(\mYs)\leq \frac{1}{2}(\dim \mYs+3)$$ so that $\mYs$ is generically smooth.  It is possible that $\mYs$ has irreducible components of dimension 0 and 2, which are necessarily smooth, if $E\cong E\otimes \omega_Y$ and $$\dim_E \mYs=v^2+2,$$ i.e. $v^2=-2$ or $0$.
\end{Thm}
\begin{proof} If $E\in \mYs$ is a singular point, then the obstruction space does not vanish, so $$\ext^2(E,E)=\hom(E,E\otimes\omega_Y)=\dim T_E \mYs-(v^2+1)>0,$$ so $E\cong E\otimes \omega_Y$, and we assume this to be the case for the time being.  Then as usual $E\cong \pi_*F$ for some $F\in M^s_{\sigma',\Y}(w)^{\circ}$.  It follows that $\pi^*E\cong F\oplus \iota^*F$.  If the Mukai vector of $F$ is $v(F)=(r,c_1(F),r+\frac{1}{2}c_1(F)^2-c_2(F))$, then the rank of $v$ is $2r$ and \begin{equation}\label{c2}c_2(E)=\frac{1}{2}c_1(F).\iota^*c_1(F)+c_2(F).\end{equation}  Choosing any ample divisor $H$ on $Y$, $\pi^*H$ is an $\iota^*$-invariant ample divisor on $\Y$, so $ c_1(F).\pi^*H=\iota^*c_1(F).\pi^*H$ and by the Hodge Index Theorem \begin{equation}\label{hodge}2(c_1(F)^2-c_1(F).\iota^*c_1(F))=(c_1(F)-\iota^*c_1(F))^2\leq 0\end{equation} with equality if and only if $\iota^*c_1(F)=c_1(F)$.  A direct computation shows that $$v(E)^2+1=2(v(F)^2+2)+(c_1(F).\iota^*c_1(F)-c_1(F)^2)-3.$$  Since $F$ is stable on the K3 surface $\Y$, the moduli space $M_{\sigma',\Y}(v(F))$ is smooth of dimension $v(F)^2+2$ at $F$.  It follows that $$\dim M_{\sigma',\Y}(v(F))\leq \left\{\begin{array}{l} \frac{1}{2}(\dim \mYs+3)\mbox{ if }\dim_E \mYs=v^2+1\\ \frac{1}{2}(\dim \mYs+2)\mbox{ if }\dim_E \mYs=v^2+2\end{array}\right\},$$ with equality if and only if $c_1(F)=\iota^*c_1(F)$.  Note that if $E$ is on a component of dimension $v^2+2$, then the entire component is smooth and $E'\cong E'\otimes\omega_Y$ for every other point $E'$ on this component.  

If $E$ is indeed a singular point, then we must have that $E$ is on a component of the expected dimension $v^2+1$.  Let $B=-\dim_E M_{Y,\sigma}(v)-3$ so that from the above inequality we must have $$c_1(F)^2-c_1(F).\iota^*c_1(F)\geq B,\text{ and}$$ $$2B\leq (c_1(F)-\iota^*c_1(F))^2\leq 0.$$  Thus there can only be finitely many numbers $(c_1(F)-\iota^*c_1(F))^2$ and for any fixed value there can be only finitely many choices for $c_1(F)$ since $(\pi^*H)^{\perp}$ is a negative definite lattice.  Since $c_2(F)$ is determined by \ref{c2} above, there are only finitely many possible Mukai vectors. 

Furthermore notice that $\pi_*F$ is $\sigma$-stable if and only if $F\ncong\iota^*F$.  Indeed, $\pi_*F$ stable implies that $$\C=\Hom(\pi_*F,\pi_*F)=\Hom(F,\pi^*\pi_*F)=\Hom(F,F)\oplus\Hom(F,\iota^*F),$$ which implies that $F\ncong\iota^*F$.  The converse follows analogously.  Thus the singular locus of $\mYs$ is the image under $\pi_*$ of finitely many $M^s_{\sigma',\Y}(v(F))^{\circ}$'s for the finitely many $v(F)$'s satisfying the above necessary requirements.

Now we show that the push-forward map $\pi_*:M^s_{\sigma',\Y}(v(F)^{\circ})\rightarrow \mYs$ is \'{e}tale of degree 2.  Suppose that $\pi_*(F)\cong\pi_*(G)$, then pulling back gives $F\oplus\iota^*F\cong G\oplus\iota^*G$, so $$\C^2=\Hom(F\oplus\iota^*F,G\oplus\iota^*G)=\Hom(F,G)^{\oplus 2}\oplus\Hom(F,\iota^*G)^{\oplus 2},$$ so either $F\cong G$ or $F\cong \iota^*G$, but not both.  Thus the singular locus is of even dimension and smooth itself.

If $E$ instead lies on a component of dimension $v^2+2$, and is thus a smooth point, then letting $C=-\dim_E \mYs-2$ we again get the same bound $$2C\leq (c_1(F)-\iota^*c_1(F))^2\leq 0,$$ and thus again the component $M$ containing $E$ is the image under $\pi_*$ of some components of $M^s_{\sigma',\Y}(v(F))^{\circ}$ for the finitely many $v(F)$ satisfying this bound.  The argument above shows that this map $\pi_*$ is finite \'{e}tale so that if $\overline{M}$ is one such component, then $$\dim M=\dim \overline{M}\leq \frac{1}{2}(\dim M+2).$$  This forces the dimension of $M$ to be 0 or 2.  If $\dim M=2$, then we get equality in this inequality so that $c_1(F)=\iota^* c_1(F)$ and thus $\v=2v(F)$ is not primitive.  If $\dim M=0$, then $M$ has a unique stable object $E$ and and $\overline{M}$ has two objects $F$ and $\iota^*F$ with $c_1(F)^2=c_1(F).\iota^*c_1(F)-2$.
\end{proof}

Having obtained some global description of the singular locus, we now generalize the results of \cite{Yamada} on the nature of these singularities and the canonical bundle to Bridgeland moduli spaces:
\begin{Thm}\label{K trivial} Suppose that $Y$ is an Enriques surface, $v\in\Hal(Y,\Z)$, and $\sigma\in\Stab^{\dagger}(Y)$.  Suppose that the fixed locus of $-\otimes\omega_Y$ has codimension at least 2.  Then $\mYs$ is normal and Gorenstein with only canonical l.c.i. singularities.  Furthermore, $\omega_{\mYs}$ is torsion in $\Pic(\mYs)$.
\end{Thm}
\begin{proof} Since we do not need the entire statement for the sequel, we only sketch the proof.  From the proof of Theorem \ref{singular}, we see that any singularities are hypersurface singularities which comprise a locus of codimension at least 2 by assumption.  Thus $\mYs$ is l.c.i. and smooth in codimension 1, so it's normal and Gorenstein as well.  That $\mYs$ has canonical singularities follows precisely as in \cite{Yamada}.

For the claim about the canonical divisor, denote by $M_0$ the open locus of objects $E$ such that $E\ncong E\otimes\omega_Y$.  Let $\mathcal E\in \Db(\mYs\times Y)$ be a quasi-universal family (which can always be constructed \'{e}tale locally), and $p:\mYs\times Y\rightarrow \mYs,q:\mYs\times Y\rightarrow Y$ the two projections.  Then $$[\lExt_p(\mathcal E,\mathcal E)]=[\lExt^0_p(\mathcal E,\mathcal E)]-[\lExt_p^1(\mathcal E,\mathcal E)]+[\lExt_p^2(\mathcal E,\mathcal E)]=p_!(\mathcal E^{\vee}\otimes \mathcal E),$$ where $\lExt_p(\EE,\EE)^i$ are the relative ext sheaves in the flat base change theorem of \cite{BPS}.  Since the objects are all stable we must have $\lExt_p^0(\mathcal E,\mathcal E)\cong \mathcal O_{\mYs}$.  Now by the Grothendieck-Riemann-Roch theorem, $$c_1([\lExt_p(\mathcal E,\mathcal E)])=c_1(p_!(\mathcal E^{\vee}\otimes\mathcal E))=\{p_*(ch(\mathcal E^{\vee}).ch(\mathcal E).q^*(td(Y)))\}_1.$$  Denoting the rank of $\mathcal E$ by $k$, we get $$ch(\mathcal E^{\vee}).ch(\mathcal E)=k^2-c_2(\mathcal E^{\vee}\otimes\mathcal E)+...,$$ where ... denotes terms of degree $\geq4$.  Of course $td(Y)=1+\frac{1}{12}c_2(Y)=1+\pt$, so combining things we get \begin{align*}ch(\mathcal E^{\vee}).ch(\mathcal E).q^*(td(Y))&=(k^2-c_2(\mathcal E^{\vee}\otimes\mathcal E)+...).(1+[\mYs\times \{\pt\}])\\&=k^2+...,\end{align*} where again ... denotes terms of degree $\geq 4$.  But upon pushing down by $p_*$, the only terms that contribute to degree 1 would be of degree 3 on $\mYs\times Y$, of which there are none.  Thus $$0=c_1([\lExt_p(\mathcal E,\mathcal E)])=-c_1([\lExt^1_p(\mathcal E,\mathcal E)])+c_1([\lExt^2_p(\mathcal E,\mathcal E)]),$$ since $c_1([\lExt^0_p(\mathcal E,\mathcal E)])=0$.  But the assumption that the complement of $M_0$ has codimension at least 2 implies that the support of $\lExt^2_p(\mathcal E,\mathcal E)$ has codimension at least 2 so that its $c_1$ vanishes.  Thus $$K_{\mYs}=-c_1(\mathcal T_{\mYs})=0,$$ since $\mathcal T_{\mYs}\cong \lExt^1_p(\mathcal E,\mathcal E)$.  Since Grothendieck-Riemann-Roch is a statement about Chow groups with $\Q$ coefficients, this shows that $K_{\mYs}$ is torsion.
\end{proof}

\begin{Rem} Recall from Theorem \ref{singular} that $M_0=\mYs$ if the rank of $v$ is odd.  As this is the case if $v^2$ is odd, the hypothesis of the theorem above is certainly satisfied in this case.  Furthermore, the codimension of the complement of $M_0$ is larger than 2 if $v^2\geq 5$ by Theorem \ref{singular} since it must be of even dimension at most $\frac{1}{2}(\dim \mYs+3)$.  Moreover, we have equality in this dimension estimate only if $c_1(F)=\iota^*c_1(F)$, in which case $v$ cannot be primitive.  From Section \ref{sec:ReviewGieseker} we know that if $\v$ is divisible by 2 for primitive $v$ then $v^2\equiv 0(\mod 8)$.  So if $v^2=4$, then the complement of $M_0$ has codimension at least 2 if $v$ is primitive, but if $v^2=2$, it is possible that the complement of $M_0$ is a divisor.  If $Y$ is unnodal, however, then $\v$ being primitive precludes this since we automatically have $c_1(F)=\iota^*c_1(F)$ in this case.  Finally, let us note that on exceptional components of dimension 2, i.e. $v^2=0$ and $E\cong E\otimes\omega_Y$, $\lExt^2_p(\mathcal E,\mathcal E)\cong \OO_{\mYs}$, so the conclusion of the theorem continues to hold.
\end{Rem}

Let us conclude this section by summarizing the consequences of the above theorem in the unnodal case.
\begin{Cor} Let $Y$ be an unnodal Enriques surface, $v=mv_0\in\Hal(Y,\Z)$ with $m\in\Z_{>0}$, $v_0^2>0$ and $v_0$ primitive.  Furthermore, let $\sigma\in\Stab^{\dagger}(Y)$ be generic with respect to $v$.  The projective variety $\mY$ of dimension $v^2+1$ is normal and $K$-trivial unless $v_0^2=1$ and $m=2$.  In particular, for $m=1$, $\mY$ is a normal projective $K$-trivial variety of dimension $v^2+1$, smooth unless $v^2\equiv 0(\mod 8)$.
\end{Cor}
\begin{proof} The only thing left to note is that by Theorem \ref{num dim} the strictly semistable locus has codimension at least 2 unless $v_0^2=1$ and $m=2$ in which case it forms a divisor.  Furthermore, in this case the complement of $M_0\subset\mYs$ has codimension 1.  Except for this case, the semistable locus has high codimension and by Theorems \ref{singular},\ref{K trivial}, and the remark above, so does the complement of $M_0\subset\mYs$.  Then the class of $K_{\mY}$ is determined by its restriction to $M_0$ where it is numerically trivial.  
\end{proof}

\section{A Natural Nef Divisor}\label{sec:BM divisor}

In this section we will again restrict ourselves to the case that $Y$ is unnodal and $v\in\Hal(Y,\Z)$ is primitive with $v^2\geq 1$, but first let $X$ be an arbitrary smooth complex projective variety and $v\in\Hal(X,\Z)$.  By Remark \ref{phi one} we may assume that $\sigma\in\Stab(X)$ is algebraic and $Z(v)=-1$.  Let us recall the definition and properties of the divisor class $\ell_{\sigma}$ \cite[Proposition and Definition 3.2]{BaMa} on a proper algebraic space $S$ of finite type over $\C$ associated to such a $\sigma$.  First we have the general definition:

\begin{PropDef}\label{divisor def}To any projective curve $C$ with a morphism $C\to \MMM_{\sigma,X}(v)$ we associate a number $\ell_{\sigma}.C$ as follows:  let $\mathcal E\in \Db(C\times X)$ be the corresponding universal family on $C$, and let $\Phi_{\mathcal E}:\Db(C)\rightarrow \Db(X)$ be the associated Fourier-Mukai transform.  Then $$\ell_{\sigma}.C:=\Im Z(\Phi_{\mathcal E}(\mathcal O_C)).$$  This has the following properties:

(a) Modifying the universal family by tensoring with the pull-back of a line bundle from $C$ does not modify $\ell_{\sigma}.C$.

(b) We can replace $\mathcal O_C$ by any line bundle on $C$ without changing $\ell_{\sigma}.C$.

\end{PropDef}

It follows from \cite[Lemma 3.3 and Theorem 4.1]{BaMa} that this association gives a well-defined nef divisor class on any proper algebraic space $S$ with a family $\EE\in\MMM_X$ which we denote by $\ell_{\sigma,\EE}$ to emphasize the dependence on the family $\EE$:
\begin{Lem} $\ell_{\sigma,\EE}.C\geq0$ for every effective curve $C\subset S$ and depends only on the numerical curve class $[C]\in N_1(S)$.  Thus $\ell_{\sigma,\EE}$ defines a nef numerical divisor class in $N^1(S)$, unchanged by tensoring the family $\EE$ with a line bundle pulled back from $S$.  Further, we have $\ell_{\sigma,\EE}.C>0$ if and only if for two general closed points $c,c'\in C$, the corresponding objects $\mathcal E_c,\mathcal E_{c'}\in \Db(X)$ are not $S$-equivalent.  
\end{Lem}

Now we recall the definition of the Donaldson morphism \cite[Section 8.1]{HL} associated to such a family $\mathcal E$ on $S$: $$\lambda_{\mathcal E}:v^{\perp}\subset K_{\num}(X)\rightarrow N^1(S),$$ defined by $$\lambda_{\mathcal E}(u):=\det((p_{S})_*([\mathcal E].(p_X)^*(u))).$$  Since the Euler characteristic is non-degenerate, for a stability condition $\sigma=(Z,\mathcal A)$ with $Z(v)=-1$ as above we can write $\Im(Z(-))=\chi(w_Z,-)$ for a unique vector $w_Z\in v^{\perp}$.  Then \cite[Proposition 4.4]{BaMa} gives the following comparison result:

\begin{Prop}\label{donaldson compare} For any integral curve $C\subset S$ $$\lambda_{\mathcal E}(w_Z).C=\Im Z(\Phi_{\mathcal E}(\mathcal O_C))=:\ell_{\sigma,\mathcal E}.C.$$
\end{Prop}

In general, there is no guarantee that the moduli spaces we have constructed are actually fine moduli spaces, and we would like to associate a nef divisor class to our coarse moduli spaces depending only on $\sigma$, as in Definition \ref{divisor def} above.  To remedy the possible lack of a universal family, Mukai \cite{Muk87} came up with the following substitute, which is usually good enough for most purposes:
\begin{Def} Let $S$ be an algebraic space of finite-type over $\C$.
\begin{enumerate} 
\item A family $\EE$ on $T\times X$ is called a \emph{quasi-family} of objects in $\MMM_{\sigma,X}(v)$ if for all closed points $t\in T$, there exists $E\in \MMM_{\sigma,X}(v)(\C)$ such that $\EE_t\cong E^{\oplus \rho}$, where $\rho>0$ is an integer which is called the \emph{similitude} and is locally constant on $T$.
\item Two quasi-families $\EE$ and $\EE'$on $T$, of similitudes $\rho$ and $\rho'$, respectively, are called \emph{equivalent} if there are locally free sheaves $\NN$ and $\NN'$ on $T$ such that $\EE\otimes p_T^*\NN\cong \EE'\otimes p_T^*\NN'$.  It follows that the similitudes are related by $\rk \NN \cdot \rho=\rk \NN'\cdot \rho'$.
\item A quasi-family $\EE$ is called \emph{quasi-universal} if for every scheme $T'$ and quasi-family $\EE'$ on $T'$, there exists a unique morphism $f:T'\to T$ such that $f^*\EE$ is equivalent to $\EE'$.
\end{enumerate}
\end{Def}

As follows from \cite[Lemma 8.1.2]{HL}, if $\NN$ is a locally free sheaf of rank $n$ on $S$, then $$\ell_{\sigma,\EE\otimes p_S^*\NN}=\lambda_{\EE\otimes p_S^*\NN}(w_Z)=n\lambda_{\EE}(w_Z)=n\cdot\ell_{\sigma,\EE}.$$ Thus, if we define $\ell_{\sigma}:=\frac{1}{\rho}\ell_{\sigma,\mathcal E}$, where $\rho$ is the similitude of $\mathcal E$, then (b) in the definition above shows that this gives a divisor class that is independent of the equivalence class of the quasi-family.  The usual techniques (see for example \cite[Theorem A.5]{Muk87} or \cite[Section 4.6]{HL}) show that a quasi-universal family exists on $M_{\sigma,X}^s(v)$ and is unique up to equivalence.  In particular, if $\sigma$ is generic and $v$ primitive, then we get a well-defined nef divisor class on $M_{\sigma,X}(v)$.

To go further in the case of Enriques surfaces, we must note the following general result relating this divisor class to the pull-back of the corresponding divisor class on the inducing variety:

\begin{Prop}\label{pull-back} Suppose $G$ acts fixed-point-freely on a smooth projective variety $X$ with $Y=X/G$ and projection $\pi$.  Set $\sigma=(\pi^*)^{-1}(\sigma')$, and denote the corresponding pull-back morphism of stacks $$\rho:\MY\rightarrow \MMM_{\sigma',X}(\v).$$  Then $$\rho^*\ell_{\sigma'}=\ell_{\sigma}.$$
\end{Prop}
\begin{proof} From the definition, it suffices to check that $$\rho^*\ell_{\sigma'}.C=\ell_{\sigma}.C$$ for every projective curve $C$ with a morphism $C\rightarrow \MY$.  For such a curve $C$ we get a universal object $\mathcal E\in \Db(C\times Y)$, and the universal object corresponding to the composition with $\rho$ is $\mathcal F=(1\times \pi)^*(\mathcal E)\in \Db(C\times X)$.  Consider the following commutative diagram 

\[\begin{CD}
C@ <p_C << C\times X@ >p_X >>X\\
@VV 1 V  @VV1\times\pi V  @VV \pi V\\
C @<p_C<<C\times Y@>p_Y >>Y
\end{CD}\] of schemes. 

Then we have \begin{align*} \rho^*\ell_{\sigma'}.C&=\ell_{\sigma'}.\rho_*(C)=\Im Z_{\sigma'}(\Phi_{(1\times\pi)^*\mathcal E}(\mathcal O_C))=\Im Z_{\sigma'}((p_X)_*(\FF\otimes (p_C)^*(\mathcal O_C)))\\
&=\Im Z_{\sigma'}((p_X)_*(\mathcal F\otimes(1\times\pi)^*(p_C)^*(\mathcal O_C))))\\
&=\Im Z_{\sigma'}((p_X)_*((1\times \pi)^*(\mathcal E\otimes (p_C)^*(\mathcal O_C))))\\
&=\Im Z_{\sigma'}(\pi^*((p_Y)_*(\mathcal E\otimes(p_C)^*\mathcal O_C))))\\
&=\Im Z_{\sigma}(\Phi_{\mathcal E}(\mathcal O_C))=\ell_{\sigma}.C,
\end{align*} as required.
\end{proof}

Finally, we may use this proposition to deduce the following result for an unnodal Enriques surface $Y$:

\begin{Thm} Let $Y$ be an unnodal Enriques surface, $v=mv_0\in\Hal(Y,\Z)$, and $\sigma\in\Stab^{\dagger}(Y)$ generic with respect to $v$, where $v_0$ is primitive and $m\in\Z_{>0}$.  Then $\ell_{\sigma}$ is ample on the projective variety $\mY$.
\end{Thm}
\begin{proof} By \cite[Corollary 7.5]{BaMa} and the discussion after it, $\ell_{\sigma'}$ is ample on $\mX$.    By Theorem \ref{projective general} it follows that the morphism $\rho$ of Proposition \ref{pull-back} (we called it $\Phi$ above in our case) is a finite morphism, and thus $\ell_{\sigma}=\Phi^*\ell_{\sigma'}$ is ample.
\end{proof}

\section{Flops via Wall-Crossing}\label{sec: wall-crossing}

With the above preparations, we can now explain the main tool we need to investigate the birational operation induced by crossing a wall $W$ in $\Stab^{\dagger}(Y)$ associated to a primitive Mukai vector $v$ with $v^2\geq 1$ on an unnodal Enriques surface $Y$.  

Let $\sigma_0=(W_0,\mathcal A_0)\in W$ be a generic point on the wall.  Let $\sigma_+=(Z_+,\mathcal A_+),\sigma_-=(Z_-,\mathcal A_-)$ be two algebraic stability conditions in the two adjacent chambers on each side of $W$.  From Section \ref{sec:SingKod}, the two moduli spaces $M_{\pm}:=M_{\sigma_{\pm},Y}(v)$ are non-empty $K$-trivial normal projective varieties, smooth outside of codimension two, and parametrize only stable objects.  Choosing (quasi-)universal families $\mathcal E_{\pm}$ on $M_{\pm}$ of $\sigma_{\pm}$-stable objects, we obtain (quasi-)families of $\sigma_0$-semistable objects from the closedness of semistability.  By \cite[Theorem 4.1]{BaMa}, these two families give two nef divisor classes $\ell_{0,\pm}:=\ell_{\sigma_0,\mathcal E_{\pm}}$ on $M_{\pm}$.

Following \cite{BaMa}, we enumerate four different possible phenomena at the wall $W$ depending on the codimension of the locus of strictly $\sigma_0$-semistable objects and the existence of curves $C\subset M_{\pm}$ with $\ell_{0,\pm}.C=0$, i.e. curves parametrizing $S$-equivalent objects.  We call the wall $W$

\begin{enumerate}
\item a \emph{fake wall} if there are no curves in $M_{\pm}$ of objects that are $S$-equivalent to each other with respect to $\sigma_0$,

\item a \emph{totally semistable wall}, if $M_{\sigma_0}^s(v)=\varnothing$,

\item a \emph{flopping wall}, if $W$ is not a fake wall and $M_{\sigma_0}^s(v)\subset M_{\pm}$ has complement of codimension at least two,

\item a \emph{bouncing wall}, if there is an isomorphism $M_+\cong M_-$ that maps $\ell_{0,+}$ to $\ell_{0,-}$, and there are divisors $D_{\pm}\subset M_{\pm}$ that are covered by curves of objects that are $S$-equivalent to each other with respect to $\sigma_0$.
\end{enumerate}

We may assume that $\sigma_0$ is algebraic, $W_0(v)=-1$, and $\phi=1$.  Then $\ell_{0,\pm}$ is the pull-back by the finite-morphism $\Phi$ of a semi-ample divisor \cite[Section 8]{BaMa}.  Thus $\ell_{0,\pm}$ is itself semi-ample (as are its restrictions to each irreducible component of $M_{\pm}$).

We thus get induced contraction morphisms \cite[Theorem 2.1.27]{Laz04} $$\pi_{\sigma_{\pm}}:M_{\pm}\rightarrow Z_{\pm},$$ where $Z_{\pm}$ are normal projective varieties.  We denote the induced ample divisor class on $Z_{\pm}$ by $\ell_0$, i.e., the ample divisor pulling back to $\ell_{0,\pm}$.  If $M_{\sigma_0,Y}^s(v)\neq \varnothing$, then by the openness of stability for primitive Mukai vectors \cite[Proposition 9.4]{Bri08} these objects remain $\sigma_{\pm}$-stable, and we denote by $f_{\sigma_0}:M_+\dashrightarrow M_-$ the induced birational map.

As observed in \cite{BaMa}, $\pi_{\sigma_{\pm}}$ is an isomorphism if and only if the wall $W$ is a fake wall, a divisorial contraction if $W$ is a bouncing wall, and a flopping contraction if $W$ is a flopping wall.

The proof of \cite[Proposition 8.1]{BaMa} carries through unchanged to yield the following result which shows that $Z_{\pm}$ is a union of components of a coarse moduli space of $\sigma_0$-semistable objects up to a finite cover:
\begin{Prop} The space $Z_{\pm}$ has the following universal property: For any proper irreducible scheme $S\in \Sch_{\C}$, and for any family $\mathcal E\in \MMM_{\sigma_0,Y}(v)(S)$ such that there exists a closed point $s\in S$ such that $\mathcal E_s\in \MMM_{\sigma_{\pm},Y}(v)(\C)$, there exists a finite morphism $q:T\rightarrow S$ and a natural morphism $f_{q^*\mathcal E}:T\rightarrow Z_{\pm}$.
\end{Prop}

Now, as $M_{\pm}$ are $K$-trivial, the existence of a $\sigma_0$-stable object of class $v$ inducing the birational map $f_{\sigma_0}$ can be extended to an isomorphism away from a locus of codimension at least 2 (see for example \cite[Proposition 3.52(2)]{KM98}).  The proof of \cite[Lemma 10.10]{BaMa} then carries through unchanged to give the following identification of $\ell_{0,+}$ with $\ell_{0,-}$:
\begin{Lem}\label{id ns} Let $Y$ be an unnodal Enriques surface and $v\in\Hal(Y,\Z)$ primitive with $v^2\geq 1$.  Assume that there exists a $\sigma_0$-stable object of class $v$ and identify the N\'{e}ron-Severi groups of $M_{\pm}(v)$ by extending the common open subset $M^s_{\sigma_0,Y}(v)$ to an isomorphism outside of codimension two.  Under this identification, $\ell_{0,+}=\ell_{0,-}$.
\end{Lem}

Finally, we have enough preparation to prove our main result about the relationship between wall-crossing and birational geometry:
\begin{Thm} \label{contraction} Let $Y$ be an unnodal Enriques surface and $v\in\Hal(Y,\Z)$ a primitive Mukai vector.
\begin{enumerate}
\item The divisor classes $\ell_{0,\pm}$ are semiample (and remain so when restricted to each component of $M_{\pm}$), and they induce contraction morphisms $$\pi_{\pm}:M_{\pm}\rightarrow Z_{\pm},$$ where $Z_{\pm}$ are normal projective varieties.

\item Suppose that $M_{\sigma_0,Y}^s(v)\neq\varnothing$.  
\begin{itemize}
\item If either $\ell_{0,\pm}$ is ample, then the other is ample, and the birational map $$f_{\sigma_0}:M_+\dashrightarrow M_-$$ obtained by crossing the wall in $\sigma_0$ extends to an isomorphism.
\item If $\ell_{0,\pm}$ are not ample and the complement of $M_{\sigma_0,Y}^s(v)$ has codimension at least 2, then $f_{\sigma_0}:M_+\dashrightarrow M_-$ is the flop induced by $\ell_{0,+}$.  More precisely, we have a commutative diagram of birational maps
\begin{equation*}
\xymatrix{ M_{\sigma_{+},Y}(v)\ar@{-->}[rr]^{f_{\sigma_0}}\ar[dr]_{\pi_+} && M_{\sigma_{-},Y}(v)\ar[dl]^{\pi_{-}}\\
& Z_+=Z_{-} &
},
\end{equation*}
and $f_{\sigma_0}^*\ell_{0,-}=\ell_{0,+}$.
\end{itemize}
\end{enumerate}
\end{Thm}
\begin{proof} 

It only remains to prove part (b).  The first case of (b) follows from Lemma \ref{id ns} and the theorem of Matsusaka and Mumford \cite[Exercise 5.6]{KSC}.

For the second case, note that by the discussion preceeding Lemma \ref{id ns}, we need not specify in which moduli space we assume the complement of the $\sigma_0$-stable locus to have codimension 2.  From this codimension condition and projectivity of these moduli spaces, it follows that numerical divisor classes are determined by their intersection numbers with curves contained in $M_{\sigma_0,Y}^s(v)$.  By Lemma \ref{id ns} we have $f_{\sigma_0}^*\ell_{0,-}=\ell_{0,+}$.  Since $\ell_{0,+}$ is not ample, $f_{\sigma_0}$ does not extend to an isomorphism.  The identification of $\ell_{0,\pm}$ and the codimension condition force $Z_+=Z_-$ from their construction in \cite[Proposition 2.1.27]{Laz04}.  This gives the claimed commutativity of the diagram and thus the description of the birational map as a flop.
\end{proof}

\begin{Rem} We believe the semiample divisors $\ell_{0,\pm}$ are big as well.  If they weren't, then every irreducible component of $M_{\pm}$ would fiber over a component of $Z_{\pm}$ with positive dimensional fibers, i.e. $\dim Z_{\pm}<\dim M_{\pm}$.  From the description of semistable objects as extensions of their stable factors, we suspect that it would follow that $M_{\pm}$ are then covered by a family of rational curves contracted by $\pi_{\pm}$.  By \cite[Remark 4.2 (4)]{Deb} $M_{\sigma_{\pm}}$ would then be uniruled, which is impossible as they are $K$-trivial \cite[Corollary 4.12]{Deb}.  Since being $K$-trivial applies to each component, the same argument shows that the restriction to each component is big as well.
\end{Rem}
\section{Moduli of stable sheaves}\label{sec: stable sheaves}

We'd like to now use the above work to setup the investigation of the birational geometry of the classical moduli spaces of sheaves on an unnodal Enriques surface.  From Lemma \ref{donaldson compare} above we see that $\ell_{\sigma}$ and $\lambda_{\EE}(w_Z)$ agree up to scaling by a positive real number.  We can define a dual version of the Donaldson morphism which is suited for the Mukai lattice.  We define the \emph{Mukai homomorphism} $\theta_v:v^{\perp}\to N^1(\mY)$ by $$\theta_v(w).C:=\frac{1}{\rho}(w,\Phi_{\EE}(\OO_C)),\text{ for every projective integral curve }C\subset \mY,$$ where $\EE$ is a quasi-universal family of similitude $\rho$ and $\Phi_{\EE}$ is the associated Fourier-Mukai transform.  The relationship between the Donaldson and Mukai homomorphisms is $$\theta_v(v(w))=-\frac{1}{\rho}\lambda_{\EE}(w^*),$$ where $w^*$ denotes the dual of $w$ in $K_{\num}(Y)_{\R}$.  We can relate $\theta_v$ and $\ell_{\sigma}$ explicitly as in \cite[Lemma 9.2]{BaMa}:

\begin{Lem}\label{donaldson}  Let $Y$ be an Enriques surface, $v=(r,c,s)$ a primitive Mukai vector with $v^2\geq-1$, and let $\sigma=\sigma_{\omega,\beta}\in \Stab^{\dagger}(Y)$ be a generic stability condition with respect to $v$.  Then the divisor class $\ell_{\sigma}\in N^1(\mY)$ is a positive multiple of $\theta_v(w_{\sigma_{\omega,\beta}})$, where $w_{\sigma_{\omega,\beta}}=(R_{\omega,\beta},C_{\omega,\beta},S_{\omega,\beta})$ is given by 
$$R_{\omega,\beta}=c.\omega-r\beta.\omega,$$
$$C_{\omega,\beta}=(s-\beta.c+r\frac{\beta^2-\omega^2}{2})\omega+(c.\omega-r\beta.\omega)\beta,\text{ and}$$
$$S_{\omega,\beta}=c.\omega\frac{\beta^2-\omega^2}{2}+s\beta.\omega-(c.\beta)(\beta.\omega).$$
\end{Lem}
\begin{proof}  For these stability conditions we have $Z(E)=(v(E),e^{\beta+i\omega})$, so for any integral curve $C\subset \mY$ \begin{align*}\ell_{\sigma}.C=\frac{1}{\rho}\Im\left(-\frac{Z(\Phi_{\EE}(\OO_C))}{Z(v)}\right)&=\frac{1}{\rho}\Im\left(-\frac{(v(\Phi_{\EE}(\OO_C)),e^{\beta+i\omega})}{(v,e^{\beta+i\omega})}\right)\\&=\frac{1}{\rho}\left(v(\Phi_{\EE}(\OO_C)),\Im\left(-\frac{e^{\beta+i\omega}}{(v,e^{\beta+i\omega})}\right)\right).\end{align*}  From the definition of $\theta_v$ it follows that the vector is given by $$w_{\sigma_{\omega,\beta}}=\Im\frac{e^{i\omega+\beta}}{-(e^{i\omega+\beta},v)}\sim_{\R^+}-\Im(\overline{(e^{i\omega+\beta},v)}\cdot e^{i\omega+\beta}),$$ where $\sim_{\R^+}$ means that the vectors are positive scalars of each other.  The lemma follows from $$e^{i\omega+\beta}=(1,\beta,\frac{\beta^2-\omega^2}{2})+i(0,\omega,\omega.\beta).$$
\end{proof}

Writing $\omega=tH$ for $H\in \NS(Y)$, we can let $t\to 0$ or $\infty$, even though these do not correspond to genuine Bridgeland stability conditions.  For Mukai vectors of Gieseker stable sheaves, we will see below that letting $t\to \infty$ often gives us a boundary of the nef cone of the Gieseker moduli space, while letting $t\to 0$ often gives us a nef divisor on a possibly different birational model.  For the sake of future use then we record the Mukai vectors sent by $\theta_v$ to these nef divisors: taking $t\to 0$ and rescaling gives the vector $w_{0\cdot H,\beta}$ with components
\begin{align*}
& R_{0\cdot H,\beta}= c.H-r\beta.H\\
& C_{0\cdot H,\beta}= \left(c.H-r\beta.H\right)\beta + \left(s-c.\beta + r\,\frac{\beta^2}{2} \right) H\\
& S_{0\cdot H,\beta}= c.H\, \frac{\beta^2}{2} + s \beta.H - (c.\beta) \cdot (\beta.H).
\end{align*}

Taking $t\to\infty$, we obtain a
vector $w_{\infty\cdot H,\beta}$ with components
\begin{align*}
& R_{\infty\cdot H,\beta}= 0\\
& C_{\infty\cdot H,\beta}= -r\, \frac{H^2}{2} H\\
& S_{\infty\cdot H,\beta}= -c.H\,\frac{H^2}{2}.
\end{align*}

We give a bound now for the walls of the ``Gieseker chamber" for any (primitive) Mukai vector $v$, i.e. the chamber for which Bridgeland stability of objects of class $v$ is equivalent to $\beta$-twisted Gieseker stability.  Fix a class $\beta\in \NS(Y)_{\Q}$, and let $\omega$ vary on a ray in the ample cone.  Given $v$ with positive rank and slope, we saw above that for $\omega\gg 0$ stable objects of class $v$ are exactly the $\beta$-twisted Gieseker stable sheaves.  Below we give explicit bounds for the Gieseker chamber that depend only on $\omega^2,\beta,$ and $v$.

\begin{Def} Given divisor classes $\beta$ and $\omega=tH$ with $H\in\Pic(Y)$ ample, and given a class $v=(r,c,s)$ with $v^2\geq -1$, we write $(r,c_{\beta},s_{\beta})=e^{-\beta}(r,c,s)$ so that $c_{\beta}=c-r\beta,s_{\beta}=r\frac{\beta^2}{2}-c.\beta+s$.  Define the $\beta$-twised slope and discrepancy of $v$ with respect to $\omega$ by $$\mu_{\omega,\beta}(v)=\frac{\omega.c_{\beta}}{r},\text{ and }\delta_{\omega,\beta}(v)=-\frac{s_{\beta}}{r}+1+\frac{1}{2}\frac{\mu_{\omega,\beta}(v)^2}{\omega^2},$$ respectively.  
\end{Def}

Note that scaling $\omega$ rescales $\mu_{\omega,\beta}$ by the same factor while leaving $\delta_{\omega,\beta}$ invariant.  Recall that a torsion-free coherent sheaf $F$ is called $\beta$-twisted Gieseker stable with respect to $\omega$ if for every proper subsheaf $0
\neq G\subset F$ we have $$\mu_{\omega,\beta}(G)\leq\mu_{\omega,\beta}(F)\text{ and }\delta_{\omega,\beta}(G)>\delta_{\omega,\beta}(F),\text{ if }\mu_{\omega,\beta}(G)=\mu_{\omega,\beta}(F).$$  This is an unravelling of the usual definition via reduced $\beta$-twisted Hilbert polynomials.  It follows that the definition of twisted Gieseker stability only depends on $H$ and not $t$.  Also, notice that the Hodge Index theorem implies that $\omega^2(c_{\beta})^2\leq (\omega.c_{\beta})^2$, so \begin{equation}\label{delta}\delta_{\omega,\beta}(v)\geq -\frac{s_{\beta}}{r}+1+\frac{c_{\beta}^2}{2r^2}=\frac{v^2}{2r^2}+1\geq 1-\frac{1}{2r^2}\geq \frac{1}{2}>0,\end{equation} where the second to last inequality follows from the assumption that $v^2\geq -1$.  Using this notation, we can write the central charge $Z_{\omega,\beta}(v)$ in terms of the slope and the discrepancy as in \cite{BaMa}: 

\begin{equation}\label{Gieseker charge}\frac{1}{r}Z_{\omega,\beta}(v)=i\mu_{\omega,\beta}(v)+\frac{\omega^2}{2}-1-\frac{\mu_{\omega,\beta}(v)^2}{2\omega^2}+\delta_{\omega,\beta}(v).\end{equation}  

For now, we fix a Mukai vector $v=(r,c,s)$ with $r>0$ and $\mu_{\omega,\beta}(v)>0.$  We have the following lemma whose proof is the same as in \cite[Lemma 9.10]{BaMa} and essentially follows from Figure \ref{fig:deltabound} and equation (\ref{Gieseker charge}):

\begin{Lem}\label{see-saw} Assume $\omega^2>1$ so that $\sigma_{\omega,\beta}$ is gauranteed to be a stability condition.  Then any Mukai vector $w\in \Hal(Y)$ with $r(w)>0$ and $0<\mu_{\omega,\beta}(w)<\mu_{\omega,\beta}(v)$ such that the phase of $Z_{\omega,\beta}(w)$ is bigger than or equal to the phase of $Z_{\omega,\beta}(v)$ satisfies $\delta_{\omega,\beta}(w)<\delta_{\omega,\beta}(v)$, as long as $\sigma_{\omega,\beta}$ is a stability condition.
\end{Lem}

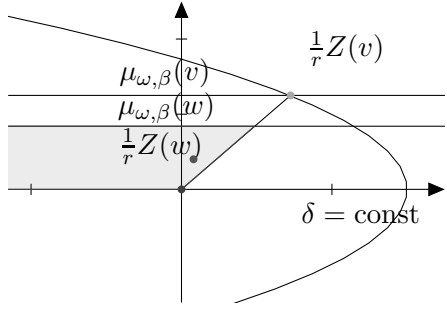
\begin{figure}
\begin{centering}
\definecolor{zzttqq}{rgb}{0.27,0.27,0.27}
\definecolor{qqqqff}{rgb}{0.33,0.33,0.33}
\definecolor{uququq}{rgb}{0.25,0.25,0.25}
\definecolor{xdxdff}{rgb}{0.66,0.66,0.66}
\begin{tikzpicture}[line cap=round,line join=round,>=triangle 45,x=1.0cm,y=1.0cm]
\draw[->,color=black] (-2.3,0) -- (3.5,0);
\foreach \x in {-2,2}
\draw[shift={(\x,0)},color=black] (0pt,2pt) -- (0pt,-2pt);
\draw[->,color=black] (0,-1.5) -- (0,2.5);
\foreach \y in {,2}
\draw[shift={(0,\y)},color=black] (2pt,0pt) -- (-2pt,0pt);
\clip(-2.3,-1.5) rectangle (3.5,2.5);
\fill[color=zzttqq,fill=zzttqq,fill opacity=0.1] (0,0) -- (0.97,0.84) -- (-2.4,0.84) -- (-2.4,0) --
cycle;
\draw [samples=50,rotate around={-270:(3,0)},xshift=3cm,yshift=0cm] plot (\x,\x*\x);
\draw (1.54,2.24) node[anchor=north west] {$ \frac 1r Z(v) $};
\draw (1.45,1.25)-- (0,0);
\draw (1.46,-0.04) node[anchor=north west] {$\delta = \textrm{const}$};
\draw [line width=0.4pt,domain=-2.3:3.5] plot(\x,{(--1.25-0*\x)/1});
\draw [domain=-2.3:3.5] plot(\x,{(--0.84-0*\x)/1});
\draw (-0.96,1.88) node[anchor=north west] {$ \mu_{\omega, \beta}(v) $};
\draw (-0.98,1.42) node[anchor=north west] {$ \mu_{\omega, \beta}(w) $};
\draw (-0.96,0.94) node[anchor=north west] {$ \frac 1r Z(w) $};
\draw [color=zzttqq] (0,0)-- (0.97,0.84);
\draw [color=zzttqq] (0.97,0.84)-- (-2.4,0.84);
\draw [color=zzttqq] (-2.4,0.84)-- (-2.4,0);
\draw [color=zzttqq] (-2.4,0)-- (0,0);
\begin{scriptsize}
\fill [color=xdxdff] (1.45,1.25) circle (1.5pt);
\fill [color=uququq] (0,0) circle (1.5pt);
\fill [color=qqqqff] (0.16,0.4) circle (1.5pt);
\end{scriptsize}
\end{tikzpicture}
\caption{Destabilizing subobjects must have smaller $\delta$}
\label{fig:deltabound}
\end{centering}
\end{figure}

We define an analogue of the set defined in \cite[Definition 9.11]{BaMa}.  Let $D_v$ be the subset of the lattice $\Hal(Y,\Z)$ defined by $$\{w:0<r(w)\leq r(v),w^2\geq-1,0<\mu_{\omega,\beta}(w)<\mu_{\omega,\beta}(v),\delta_{\omega,\beta}(w)<\delta_{\omega,\beta}(v)\}.$$   The discussion there extends to show that $D_v$ is finite and depends on $H$ but not $t$.  Furthermore, they define $$\mu^{\max}(v):=\max(\{\mu_{\omega,\beta}(w):w\in D_v\}\cup\{\frac{r(v)}{r(v)+1}\mu_{\omega,\beta}(v)\}).$$  We can use this definition obtain an effective lower bound for the Gieseker chamber:

\begin{Lem}\label{Gieseker chamber} Let $E$ be a $\beta$-twisted Gieseker-stable sheaf with $v(E)=v$.  If $$\omega^2>1+\frac{\mu^{\max}(v)}{\mu_{\omega,\beta}-\mu^{\max}(v)}\delta_{\omega,\beta}(v)+\sqrt{\left(1+\frac{\mu^{\max}(v)}{\mu_{\omega,\beta}-\mu^{\max}(v)}\delta_{\omega,\beta}(v)\right)^2-\mu^{\max}(v)\mu_{\omega,\beta}(v)}$$ then $E$ is $Z_{\omega,\beta}$-stable.
\end{Lem}
\begin{proof}  Although the proof only differs slightly from that of \cite[Lemma 9.13]{BaMa}, we explain it in full.  

Consider a destabilizing short exact sequence $A\into E\onto B$ in $\AA(\omega,\beta)$ with $\phi_{\omega,\beta}(A)\geq \phi_{\omega,\beta}(E)$.  From the long exact sequence on cohomology, it follows that $A$ is a sheaf.  Consider the HN-filtration of $A$ with respect to $\mu_{\omega,\beta}$-slope stability in $\Coh X$, $$0=\HN^0(A)\subset \HN^1(A)\subset...\subset\HN^n(A)=A,$$ and let $A_i=\HN^i/\HN^{i-1}$ be its HN-filtration factors.  From the definition of $\AA(\omega,\beta)$ it follows that $\mu_{\omega,\beta}(A_i)>0$ for all $i$.  Since the kernel of $A\to E$, $\HH^{-1}(B)$, lies in $\FF(\omega,\beta)$, we see that $\mu_{\omega,\beta}(A_i)\leq \mu_{\omega,\beta}(A_1)\leq \mu_{\omega,\beta}(v)$.  Indeed, if $i$ is minimal such that $\HN^i(A)$ has nonzero image in $E$, then $A_i$ admits a nontrivial morphism to $E$ and thus $\mu_{\omega,\beta}(A_i)\leq\mu_{\omega,\beta}(E)$.  Suppose $i>1$, then $\HN^1(A)=A_1$ maps to zero in $E$ and thus is contained in $\HH^{-1}(B)$.  If $j$ is minimal such that $A_1\subset \HN^j(\HH^{-1}(B))$, then it has nonzero image in the $j$-th HN-filtration factor of $\HH^{-1}(B)$, a contradiction since this has $\mu_{\omega,\beta}\leq 0$ from the definition of $\FF(\omega,\beta)$.  

Since $\phi_{\omega,\beta}(A)\geq \phi_{\omega,\beta}(E)$, we can choose some $i$ such that $\phi_{\omega,\beta}(A_i)\geq\phi_{\omega,\beta}(E)$ by the see-saw property.  We show that $\mu_{\omega,\beta}(A_i)<\mu_{\omega,\beta}(E)$.  If not, then $\mu_{\omega,\beta}(A_i)=\mu_{\omega,\beta}(E)$, so $i=1$.  Consider the composition $A_1\into A\to E$ with kernel $K$.  Then if $K\neq 0$, $\mu_{\omega,\beta}(K)=\mu_{\omega,\beta}(E)$.  But $K\subset \HH^{-1}(B)$, so as above we get a contradiction to the fact that $\HH^{-1}(B)\in\FF(\omega,\beta)$, and thus $K=0$.  But then $\mu_{\omega,\beta}(A_1)=\mu_{\omega,\beta}(E)$ and $\phi_{\omega,\beta}(A_1)\geq\phi_{\omega,\beta}(E)$ imply that $\Re Z(A_1)\leq\Re Z(E)$, so from equation (\ref{Gieseker charge}) we see that $\delta_{\omega,\beta}(A_1)\leq \delta_{\omega,\beta}(E)$, contradicting the $\beta$-twisted Gieseker stability of $E$.

Let $w$ be the primitive generator of the positive ray spanned by $v(A_i)$.  Then $\mu_{\omega,\beta}(w)=\mu_{\omega,\beta}(A_i)$, and from Lemma \ref{see-saw} and the definition of $D_v$ it follows that if $r(w)\leq r(v)$ we have $\mu_{\omega,\beta}(w)\leq \mu^{\max}(v)$.  If $r(w)\geq r(v)+1$, notice that $$\omega.c_{\beta}(w)\leq \omega.c_{\beta}(A_i)=\Im Z(A_i)\leq \Im Z(A)\leq \Im Z(E)=\omega.c_{\beta}(v),$$ so $\mu_{\omega,\beta}(w)\leq \frac{r(v)}{r(v)+1}\mu_{\omega,\beta}(v)\leq \mu^{\max}(v)$ in this case as well.

Consider the complex number $$z:=i\mu^{\max}(v)+\frac{\omega^2}{2}-1-\frac{\mu^{\max}(v)^2}{2\omega^2}.$$  Then it follows that $\phi_{\omega,\beta}(w)\leq \phi(z)$ from Lemma \ref{see-saw} with $Z(v)$ replaced with $z$.  Now $\Im \frac{\mu_{\omega,\beta}(v)}{\mu^{\max}(v)}z=\Im \frac{1}{r(v)}Z(v)$, and it is easy to see that $\Re \frac{\mu_{\omega,\beta}(v)}{\mu^{\max}(v)}z>\Re \frac{1}{r(v)}Z(v)$ for $\omega$ with $\omega^2$ as in the hypothesis.  Thus $\phi(z)<\phi_{\omega,\beta}(v)$.  This gives the contradiction $$\phi_{\omega,\beta}(E)\leq \phi_{\omega,\beta}(A_i)\leq \phi(z)<\phi_{\omega,\beta}(E).$$
\end{proof}

\begin{Rem}  As observed right before equation \ref{Gieseker charge}, $\delta_{\omega,\beta}(w)\geq\frac{1}{2}$, so we can in fact replace the complex number $z$ in the proof above by $$z:=i\mu^{\max}(v)+\frac{\omega^2}{2}-1-\frac{\mu^{\max}(v)^2}{2\omega^2}+\frac{1}{2}$$ to obtain a sharper bound.  We leave the necessary modifications to the reader as the bound above is usually an over estimate and in individual applications one can do better.
\end{Rem}

Nevertheless, the significance of this result is that it gives a lower bound on the $t$ required to ensure that $M_{\sigma_{tH,\beta},Y}(v)\cong M^{\beta}_H(v)$, the moduli space of ($\beta$-twisted) Gieseker stable sheaves.  Since we have an ample divisor on this moduli space, given by $\ell_{\sigma_{tH,\beta}}$, we get an explicit line segment of the ample cone.  This argument gives us the 

\begin{Cor} Let $v\in\Hal(Y,\Z)$ be primitive of positive rank with $v^2\geq -1$.  Let $\omega,\beta\in \NS(Y)_{\Q}$ be generic with respect to $v$ such that $\omega.c_{\beta}(v)>0$.  If $$\omega^2>1+\frac{\mu^{\max}(v)}{\mu_{\omega,\beta}-\mu^{\max}(v)}\delta_{\omega,\beta}(v)+\sqrt{\left(1+\frac{\mu^{\max}(v)}{\mu_{\omega,\beta}-\mu^{\max}(v)}\delta_{\omega,\beta}(v)\right)^2-\mu^{\max}(v)\mu_{\omega,\beta}(v)},$$ then $$\theta_v(w_{\omega,\beta})\subset \Amp(M^{\beta}_H(v)).$$
\end{Cor}

Since all Enriques surfaces have the same lattice, this gives a universal bound for all Enriques surfaces provided we've shown that the Bridgeland moduli space has a projective coarse moduli space.

\section{Hilbert Schemes of points on an Enriques Surface}\label{sec: hilbert scheme}

In this section we apply the techniques developed above to determine the nef cone of $Y^{[n]}$ for unnodal $Y$.  Let $v=(1,0,\frac{1}{2}-n)$.  Then as above we consider an ample divisor $H\in \Pic(Y)$ on an unnodal Enriques surface $Y$ with K3 cover $\tilde{Y}$, and let $\omega=tH$ and $\beta\in \NS(Y)_{\Q}$ with $t>0$.  We remark that since $Y$ is unnodal, so is $\tilde{Y}$, and thus the ample cone of $Y$ is the connected component of the round cone $D^2>0$ containing an ample divisor.  It follows that the nef cone (and consequently the effective cone since $Y$ is unnodal) is the closure of this cone, given by $D^2\geq 0$.  For $t\gg 0$ and $\beta.H<0$, $M_{t,\beta}(v):=M_{\sigma_{\omega,\beta},Y}(v)=Y^{[n]}$, the Hilbert scheme of $n$ points on $Y$.  It is well known that $Y^{[n]}$ is a smooth irreducible projective variety of dimension $2n$, and the Hilbert-Chow morphism $h:Y^{[n]}\rightarrow Y^{(n)}$ to the $n$-th symmetric product of $Y$ is a crepant resolution of singularities \cite{F1}, and since $Y$ is a regular surface, i.e. $H^1(Y,\mathcal O_Y)=0$, $\Pic(Y^{[n]})\cong\Pic(Y)\times \Z$ \cite[Corollary 6.3]{F2}.  This identification can be described explicitly as follows:  let $L^{(n)}=\psi_*(\otimes_{i=1}^n pr_i^*(L))^{S_n}$, where $\psi:Y^n\rightarrow Y^{(n)}$ is the quotient map and $pr_i:Y^n\rightarrow Y$ is the $i$-th projection.  Then $\Pic(Y^{(n)})\cong \Pic(Y)$ via this identification, and thus $\Pic(Y^{[n]})$ is generated by $h^*\Pic(Y^{(n)})$ and the divisor class $B$ where $2B$ is the exceptional divisor parametrizing non-reduced 0-dimensional subschemes of length $n$ in $Y$.  For an ample divisor $H\in \Pic(Y)$ denote by $\tilde{H}$ the corresponding divisor on $Y^{[n]}$, which is nef and big but not ample, as the pull-back of an ample divisor under the projective birational morphism $h$.  Thus $\widetilde{\Nef(Y)}$ forms an entire face of the nef cone of $Y^{[n]}$.  

We'd like to apply the results and techniques of the preceeding sections to study the birational geometry of $Y^{[n]}$.  We have the following easy first result (see \cite[Example 9.1]{BaMa} for the corresponding discussion for K3 surfaces):

\begin{Prop} The (closure of the) wall consisting of stability conditions $\sigma_{tH,\beta}$, as $H$ ranges in $\Amp(Y)$ and $\beta\in H^{\perp}$ for each fixed $H$, is one wall of the Gieseker chamber $\CC_G$.  Moreover, it is a bouncing wall sent by $\ell:\overline{\mathcal C}_G\rightarrow \Nef(Y^{[n]})$ to the wall $\widetilde{\Nef(Y)}$ above.
\end{Prop}
\begin{proof} Indeed, we have $$Z_{tH,\beta}(I_Z)=(e^{\beta+itH},1+(\frac{1}{2}-n)[pt])=(n-\frac{1}{2})+t^2\frac{H^2}{2}-\beta^2\in\R_{>0},$$ for any 0-dimensional subscheme $Z$ of length $n$ since $\beta.H=0$ implies $\beta^2\leq 0$ by the Hodge Index Theorem, so $I_Z\notin \AA(\omega,\beta)$ but $I_Z[1]$ is.  We have the following short exact sequence in $\mathcal P(1)$, $$0\rightarrow \mathcal O_Z\rightarrow I_Z[1]\rightarrow \mathcal O_Y[1]\rightarrow 0$$ which makes $I_Y$ strictly semistable.  Notice that  by filtering $\mathcal O_Z$ by structure sheaves of closed points, we see that $I_Z$ is $S$-equivalent to $\mathcal O_Y[1]\oplus \bigoplus_{i=1}^r (k(p_i))^{\oplus m_i}$, where $p_i$ are the closed points appearing in the support of $Z$ with multiplicities $m_i$.  It follows that $I_Z$ and $I_{Z'}$ are $S$-equivalent if and only if $Z,Z'$ get mapped to the same point by $h$.  Thus $\ell_{tH,\beta}$ contracts precisely the fibers of $h:Y^{[n]}\rightarrow Y^{(n)}$.  It follows that $\ell_{tH,\beta}=h^*A$ for some ample divisor $A$ on $Y^{(n)}$.  That $\ell_{tH,\beta}$ is in fact $\tilde{H}$ (or at least $\sim_{\R^+}\tilde{H}$) follows from \cite[Examples 8.2.1 and 8.2.9]{HL} and Lemma \ref{donaldson}.

Crossing this wall, i.e. taking $\beta$ such $1\gg\beta.H>0$, does not change the moduli space, i.e. $M_{t,\beta}(v)\cong Y^{[n]}$, but it does change the universal family by replacing $I_Z$ with its derived dual $\RlHom(I_Z,\OO_Y)[1]$ (see \cite[Theorem 3.1, Lemma 3.2]{Mar}).  Following a path from a point such that $\beta.H<0$ to one with $\beta.H>0$ causes the nef divisor $\ell_{tH,\beta}$ to hit the wall $\widetilde{\Nef(Y)}$ and bounce back into the interior of the ample cone.
\end{proof}

The above behavior is common for a wall inducing a divisor contraction, hence the name ``bouncing wall."  Before we describe further wall-crossing behavior, let us first point out a simple fact that is very helpful:

\begin{Lem} \label{rank one subobjects}  Let $0\to E\to I_Z\to Q\to 0$ be a non-trivial short exact sequence in $\AA_{t,b}$.  Then $E$ is a forsion free sheaf, $\HH^0(Q)$ is a quotient of $I_Z$ of rank 0, and the kernel of $I_Z\to \HH^0(Q)$ is an ideal sheaf $I_{Z'}(-D)$ for some effective curve $D$ and some zero-dimensional scheme $Z'$.
\end{Lem}
\begin{proof} As above, we consider the long exact sequence in cohomology to see that $E$ must be a sheaf fitting into the exact sequence $$0\to \HH^{-1}(Q)\to E\to I_Z\to \HH^0(Q)\to 0.$$  If $\HH^0(Q)$ had rank one, then it would have to be equal to $I_Z$, as it's torsion-free, and then we'd get that $\HH^{-1}(Q)=E$, which is only possible if they are both 0, since $\HH^{-1}(Q)\in\FF(\omega,\beta)$ while $E\in \TT(\omega,\beta)$, contrary to the assumption of non-triviality.  Thus $\HH^0(Q)$ is a quotient of $I_Z$ of rank 0, so its kernel must be of the form claimed in the lemma.  Since $I_{Z'}(-D)$ and $\HH^{-1}(Q)$ are both torsion-free, $E$ must also be torsion-free.
\end{proof}

Above we fixed $t$ and $H$ and varied $\beta$ across the hyperplane $H^{\perp}$ in $N^1(Y)$.  Now fix $H$ with $H^2=2d$ and $k\in\Z_{>0}$ such that $k^2\leq 2d$.  To simplify matters we consider stability conditions $\sigma_{t,b}:=\sigma_{tH,bH}$ in the real 2-dimensional slice of $\Stab^{\dagger}(Y)$ represented by the upper half-plane $\{(b,t)|b\in\R,t\in\R_{>0}\}$.  It is well-known (see \cite[Section 2]{Mac}) that pseudo-walls corresponding to possibly destabilizing subobjects intersect this plane in nested semi-circles with centers along the $b$-axis.  Recall that on an Enriques surface $Y$ one defines for any $D\in\Pic(Y)$ with $D^2>0$, $$\phi(D)=\inf\{|D.F|:F\in\Pic(Y),F^2=0,F\neq 0\},$$ as in \cite[Section 2.7]{CD}, where it is shown that $\phi(D)\leq \sqrt{D^2}$.  Now we are ready to prove our main theorem about $\Nef(Y^{[n]})$:

\begin{Thm}\label{nef hilbert} Let $Y$ be an unnodal Enriques surface and $n\geq 2$.  Then $\tilde{D}-aB\in\Nef(Y^{[n]})$ if and only if $D\in\Nef(Y)$ and $0\leq na\leq D.F$ for every $0<F\in\Pic(Y)$ with $F^2=0$, or in other words $0\leq a\leq \frac{\phi(D)}{n}$. Moreover, the face given by $a=0$ induces the Hilbert-Chow morphism, and for every ample $H\in\Pic(Y)$, $\tilde{H}-\frac{\phi(H)}{n}B$ induces a flop.
\end{Thm}
\begin{proof}
Consider $0<F\in\Pic(Y)$ with $F^2=0$ and $H.F=k$, so in particular $k\geq \phi(H)$.  Set $b=-\frac{k+\epsilon}{2d}$ for $0<\epsilon\ll1$ and irrational so that there exists no weakly-spherical object $S$ such that $\Im Z_{t,b}(S)=0$.  It follows that $\sigma_{t,b}$ is a stability condition for all $t>0$.  By considering the equation of the pseudo-wall corresponding to $\OO(-F)$, we see that $I_Z$ and $\OO(-F)$ have the same phase for $Z_{t_0,b}$ where $t_0:=\frac{1}{2d}\sqrt{2d-k^2}+O(\epsilon)$.  Moreover, $\phi_{t,b}(\OO(-F))<\phi_{t,b}(I_Z)$ for $t>t_0$.

We claim that any $I_Z$ is stable for $t>t_0$.  As usual, consider a destabilizing subobject $E$ as in Lemma \ref{rank one subobjects}, and as in the proof of Lemma \ref{Gieseker chamber}, we consider the semistable factors appearing in its $\HN$-filtration with respect to $\mu_{t,b}$-slope stability.  Take one of them, say $A_i$, with $v(A_i)=(r,C,s)$ and $r>0$.  Then $$0<t(H.C+r(k+\epsilon))=\Im Z_{t,b}(A_i)<\Im Z_{t,b}(I_Z)=t(k+\epsilon),$$ from which it follows that $$-rk<H.C<-rk+k,$$ and thus $$0<H.C+rk<k.$$  The stable factors of $A_i$ have $v^2\geq -1$ and rank at least one.  Thus $\delta_{t,b}(A_i)\geq\frac{1}{2}$ from (\ref{delta}).   From (\ref{Gieseker charge}) we see that $$\Re Z_{t,b}(A_i)\geq rdt^2-\frac{r}{2}-\frac{k^2}{4dr}+O(\epsilon),$$ so that $$\phi_{t,b}(A_i)<\frac{t(k+r\epsilon)}{rdt^2-\frac{r}{2}-\frac{k^2}{4dr}+O(\epsilon)}=:\phi_0(t).$$  One can easily check that $\phi_0(t)<\phi_{t,b}(I_Z)$ for $t>t_0$.  

Thus we are reduced to considering the objects $\OO(-F)$ for effective $F\in\Pic(Y)$ with $F^2=0$ and $H.F=k$ and $k\geq \phi(H)$.  The largest such value of $t_0$ occurs for $k$ minimal, i.e. $k=\phi(H)$, so we assume this to be the case.  For those $Z$ admitting a morphism $\OO(-F)\to I_Z$, we see that the exact sequence $$0\to \OO(-F)\to I_Z\to \OO_F(-Z)\to 0,$$ destabilizes $I_Z$ at $t_0$.  The locus of such $Z$ is isomorphic to $F^{[n]}$ of dimension $n$ and can be described as the Brill-Noether locus consisting of those $Z$ such that $h^0(I_Z(F))>0$ (and thus necessarily equal to 1).  We get another (disconnected) component of the strictly semistable locus by considering the adjoint half-pencil $F+K_Y$ (i.e. the other double fiber of the elliptic fibration induced by $|2F|$) and those $Z$ with $h^0(I_Z(F+K_Y))\neq 0$ which get destabilized by the corresponding exact sequence $$0\to\OO(-F-K_Y)\to I_Z\to \OO_{F+K_Y}(-Z)\to 0.$$  Similarly, for any of the other finitely many half-pencils $F'$ such that $H.F'=k$ we get two additional components of the strictly semistable locus which must necessarily intersect the above components.  Indeed, for two half-pencils $F$ and $F'$ with $H.F=H.F'=k$, we see that $-2F.F'=(F-F')^2\leq 0$, and we get strict inequality unless $F'=F$ or $F+K_Y$.  Choosing the $n$ points to lie on this non-empty intersection $F\cap F'$ (with multiplicities if necessary), we see that the corresponding Brill-Noether loci intersect.  

To see what the contracted curves are, we first observe that $v(\OO_F(-Z))=(0,F,-n)$ is primitive since $F$ is, and one can show that $\OO_F(-Z)$ is stable.  Line bundles are always stable on an unnodal surface by \cite[Theorem 5.4 and Proposition 6.3]{AM14}, so $\OO(-F)$ is also stable in our case.  Then it follows that for $Z,Z'\in F^{[n]}$, $I_Z$ and $I_{Z'}$ are $S$-equivalent if and only if $\OO_F(-Z)\cong\OO_F(-Z')$, i.e. $Z$ and $Z'$ are linearly equivalent divisors on $F$.  Thus $\pi_+$ in Theorem \ref{contraction} contracts precisely the fibers of the classical Abel-Jacobi morphism $F^{[n]}\to \Jac^n(F)\cong F$ (since $F$ is a smooth elliptic curve) which associates to an effective divisor of degree $n$ on $F$ its associated line bundle.  Crossing the wall then induces a flop with exceptional locus of codimension $n$.  As $\OO(-F)$ and $\OO_F(-Z)$ are stable of the same phase along this wall, we see that $$\Ext^1(\OO(-F),\OO_F(-Z))=((1,-F,\frac{1}{2}),(0,F,-n))=n,$$ so by \cite[Section 4]{Mar} it follows that the objects $I_Z$ with $Z\in F^{[n]}$ are replaced by non-trivial extensions $$0\to \OO_F(-Z)\to E\to \OO(-F)\to 0$$ after crossing the wall.

Now we can plug $\omega=t_0 H, \beta=-\frac{k+\epsilon}{2d},r=1,c=0,s=\frac{1}{2}-n$ into the formulas from Lemma \ref{donaldson}, and letting $\epsilon\to 0$ we see that $w_{t_0\cdot H,-1/2d}\sim_{\R_+} (1,-nH,n-\frac{1}{2})$.  As $\theta_v(1,0,n-\frac{1}{2})=-B$ and $\theta_v(0,-H,0)=\tilde{H}$, we get $\theta_v(w_{t_0\cdot H,-1/2d})\sim_{\R_+}\tilde{H}-\frac{\phi(H)}{n}B$ and $\theta_v(w_{\infty\cdot H,-1/2d})\sim_{\R_+}\tilde{H}$.  From the above discussion we see that both of these are extremal in the nef cone, with the first ray corresponding to a flop and the second ray to the Hilbert-Chow morphism.  

The statement of the theorem follows from the above discussion and the density of rational rays in $N^1(Y^{[n]})$.
\end{proof}

\begin{Rem} One direction of the above theorem is more elementary and can be obtained directly as follows.  For any effective curve $F$, fix $p_1,...,p_{n-1}$ distinct points not on $F$ and consider the curve $$C_F(n)=\{Z\in Y^{[n]}|Z=\{p_1,...,p_{n-1},p\}\}\subset Y^{[n]},$$ where the point $p$ varies along $F$.  Then $\tilde{D}.C_F(n)=D.F$ and $B.C_F(n)=0$, so $\tilde{D}-aB\in\Nef(Y^{[n]})$ implies $D$ is nef by pairing with $C_F(n)$ for all effective $F$.  Denote by $C(n)$ the generic fiber of the Hilbert-Chow morphism.  Then $\tilde{D}.C(n)=0$ and $B.C(n)=-1$, so we see that $\tilde{D}-aB\in\Nef(Y^{[n]})$ implies $a\geq 0$.  Finally, for any half-pencil $F$, consider a pencil of degree $n$ effective divisors on $F$ and the corresponding $g^1_n:F\to\P^1$.  ${g^1_n}^*\OO_{\P^1}(-x)$ for $x\in\P^1$ gives a curve $R_F(n)$ on $Y^{[n]}$ consisting of those objects sitting in short exact sequences $$0\to\OO(-F)\to I_Z\to {g^1_n}^*(-x)\to 0.$$  Riemann-Hurwitz gives that the ramification divisor of $g^1_n$ has degree $2n$.  This is precisely $2B.R_F(n)$, and $\tilde{D}.R_F(n)=D.F$.  Then $\tilde{D}-aB\in\Nef(Y^{[n]})$ implies that $na\leq D.F$.

The reverse direction is the more surprising one.  Indeed, the fact that the nef cone is not strictly smaller than the above upper bounds is a beautiful manifestation of the fact that the half-pencils control the geometry of Enriques surfaces.
\end{Rem}

\section{Applications to linear systems} \label{sec: linear systems}

Following an idea from \cite{AB13}, we can apply the above theorem to study linear systems on $Y$ itself.  The first result in this direction is the following:

\begin{Prop}\label{vanishing} Let $Y$ be an unnodal surface and $H\in\Pic(Y)$ ample with $H^2=2d$.  Then for any $Z\in Y^{[n]}$, $$H^i(Y,I_Z(H+K_Y))=0,\text{ for }i>0,$$ provided that $$1\leq n< \frac{d\cdot \phi(d)}{2d-\phi(d)}.$$
\end{Prop}
\begin{proof} First notice that \begin{align*}H^i(Y,I_Z(H+K_Y))&\cong\Ext^i(\OO_Y,I_Z(H+K_Y))\\&\cong\Ext^{i+1}_{\Db(Y)}(\OO_Y[1],I_Z(H+K_Y))\cong\Ext^{1-i}_{\AA(\omega,\beta)}(I_Z(H),\OO_Y[1])^{\vee},\end{align*} as long as both $I_Z(H)$ and $\OO_Y[1]$ are both in $\AA(\omega,\beta)$.  Consider again the upper half $(b,t)$-plane representing stability conditions with $\omega=tH,\beta=bH$.  Then for $0\leq b<1$ and $t>0$, $I_Z(H)$ and $\OO_Y[1]$ are both in $\AA_{t,b}$ so we automatically get $H^2(Y,I_Z(H+K_Y))=0$.  Furthermore, we get $H^1(Y,I_Z(H+K_Y))^{\vee}$ is identified with $\Hom_{\AA_{t,b}}(I_Z(H),\OO_Y[1])$, and this vanishes if we can choose $b\in [0,1)$ and $t>0$ such that both $I_Z(H)$ and $\OO_Y[1]$ are $\sigma_{t,b}$-stable and $\phi_{t,b}(I_Z(H))\geq\phi_{t,b}(\OO_Y[1])$.  Of course, $\OO_Y[1]$ is always stable for $b,t>0$ by \cite[Proposition 6.3]{AM14}.  From the proof of Theorem \ref{hilbert nef cone}, we know that $I_Z$ is stable above the wall corresponding to the destabilizing subobject $\OO(-F)$, for a half-pencil $F$ with $H.F=\phi(H)$.  From the formulas given for pseudo-walls on arbitrary surfaces in \cite[Section 2]{Mac}, this wall is given by $$\left(b+\frac{n}{\phi(H)}\right)^2+t^2-\frac{1-2n}{2d}-\frac{n^2}{\phi(H)^2}=0,t>0$$ in the $(b,t)$-plane.  By \cite[Section 3]{AM14}, this means that $I_Z(H)$ is stable above the wall given by $$\left(b-1+\frac{n}{\phi(H)}\right)^2+t^2-\frac{1-2n}{2d}-\frac{n^2}{\phi(H)^2}=0,t>0.$$  

Now consider the pseudo-wall corresponding to when $\phi_{t,b}(I_Z(H))=\phi_{t,b}(\OO_Y[1])$, given by the equation $$1-2 b^2 d+2 b (d-n)-2 d t^2=0,t>0.$$  One can easily see that these semi-circles are either nested, coincide, or disjoint, and they intersect the line $b=\frac{1}{2}$ for $$t_0=\frac{\sqrt{2 d n+\phi(H)-\frac{d \phi(H)}{2}-2 n \phi(H)}}{\sqrt{2} \sqrt{d} \sqrt{\phi(H)}},t_1=\frac{\sqrt{1+\frac{d}{2}-n}}{\sqrt{2} \sqrt{d}},$$ respectively.  Then $t_1\geq t_0\geq 0$ or $t_1>0$ and $t_0\notin\R$ guarantee that both $I_Z(H)$ and $\OO_Y[1]$ are $\sigma_{t_1,1/2}$-stable and $\phi_{t_1,1/2}(I_Z(H))=\phi_{t_1,1/2}(\OO_Y[1])$, as required.  These conditions combined are equivalent to $$1\leq n<\frac{d\cdot\phi(H)}{2d-\phi(H)},$$ as required.
\end{proof}

This immediately allows us to recover some classical results about linear systems on unnodal Enriques surfaces (see \cite[Theorems 4.4.1 and 4.6.1]{CD}):
\begin{Cor}\label{lin sys} Let $Y$ be an unnodal Enriques surface and $H\in\Pic(Y)$ ample with $H^2=2d$.  Then
\begin{enumerate}
\item The linear system $|H|$ is base-point free if and only if $\phi(H)\geq 2$,
\item If $|H|$ is very ample, then $\phi(H)\geq 3$.  Conversely, if $\phi(H)\geq 4$ or $\phi(H)=3$ and $H^2=10$, then $|H|$ is very ample.
\item The linear system $|2H|$ is base-point free and $|4H|$ is very ample.
\end{enumerate}
\end{Cor}
\begin{proof} Base-point freeness and very ampleness are equivalent to the surjectivity of the restriction map $$H^0(Y,\OO_Y(H))\to H^0(Z,\OO_Z(H))$$ as $Z$ ranges over all 0-dimensional subschemes of length 1 and 2, respectively.  Since $H$ ample implies the vanishing of $H^1(Y,\OO_Y(H))$ by \cite[Theorem 1.3.1]{CD}, this is equivalent to the vanishing of $H^1(Y,I_Z(H))$ in each case.  Although the easy directions of both (a) and (b) are classical and elementary (see \cite[Theorem 4.4.1(i) and Lemma 4.6.1]{CD}), we include their proofs here for completeness.

For (a), suppose $\phi(H)=1$, and let $F$ be a half-pencil such that $H.F=1$ and $Z$ be the reduced point of intersection $H\cap F$.  Then $|H|$ restricted to $F$ is a degree 1 linear system on the elliptic curve $F$, so $h^0(H|_F)=1$.  Thus $Z$ is a base-point of $|H|$, showing the easy direction of (a).  The converse follows from Proposition \ref{vanishing} as $$\frac{d\cdot\phi(d)}{2d-\phi(d)}>1$$ if $\phi(d)\geq2$.

The easy direction above also shows that $|H|$ cannot be very ample if $\phi(H)=1$.  To finish the easy direction of (b), suppose that $\phi(H)=2$.  If $H^2=4$, then $h^0(H)=3$, so $|H|$ induces a morphism of degree 4 onto $\P^2$, which is clearly not an embedding.  If $H^2\geq 6$, then we may choose a half-pencil $F$ with $H.F=2$ so that $(H-F)^2=H^2-4>0$.  Since $(H-F).F=2>0$, it follows that $H-F$ is also ample and thus that $H^1(Y,\OO_Y(H-F))=0$.  From the exact sequence $$0\to\OO_Y(H-F)\to\OO_Y(H)\to\OO_F(H)\to 0,$$ we see that $$h^0(\OO_Y(H-F))=h^0(\OO_Y(H))-h^0(\OO_F(H))=h^0(\OO_Y(H))-2,$$ since $|H|_F|$ is a degree 2 linear system on the elliptic curve $F$.  It follows that $|H|$ induces a degree 2 map of $F$ onto a line, so $|H|$ is not very ample.  The converse again follows directly from Proposition \ref{vanishing}

Part (c) follows immediately from parts (a) and (b).
\end{proof}

\begin{Rem}\label{full result} It is important to note that Corollary \ref{lin sys} is a weakening of the classical theory of linear systems on unnodal Enriques surfaces.  Indeed, \cite[Theorem 4.4.1]{CD} states that $|H|$ is very ample for all $H$ with $\phi(H)\geq 3$ without the degree restriction we impose above.  It thus follows that even $|3H|$ is very ample.  We believe the Bridgeland stability techniques employed above can be used to recover the remaining cases.  All that is required is a more careful analysis of what occurs at and beyond the first wall of the nef cone of $Y^{[n]}$.  At the wall, Theorem \ref{nef hilbert} describes what destabilizes $I_Z$, so it is possible in theory to determine precisely for what strictly semistable $I_Z$, if any, $\Hom_{\AA_{t,b}}(I_Z(H),\OO_Y[1])$ fails to vanish.  
\end{Rem}

We can also obtain some new results about $n$-very ample line bundles.  Recall that a line bundle $\OO_Y(H)$ is called $n$-very ample if the restriction map $$\OO_Y(H)\to\OO_Z(H)$$ is surjective for every 0-dimensional subscheme  $Z$ of length $n+1$.  Proposition \ref{vanishing} immediately gives the following result:
\begin{Cor} Let $Y$ be an unnodal Enriques surface and $H\in\Pic(Y)$ ample with $H^2=2d$.  Then $\OO_Y(H)$ is $n$-very ample provided that $$0\leq n\leq \frac{d\cdot\phi(H)}{2d-\phi(H)}-1.$$
\end{Cor}

\end{document}